\documentclass[10pt]{article}
\usepackage{amssymb,amsfonts,amsmath,amsthm,amsbsy}%
\usepackage[dvips]{graphicx,color}%
\newcommand{\figdraft}{false}%
\newcommand{\figfile}[1]{#1}%
\theoremstyle{plain}%
\newtheorem{theorem}{Theorem}[section]%
\newtheorem{corollary}[theorem]{Corollary}%
\newtheorem{lemma}[theorem]{Lemma}%
\newtheorem{assumption}[theorem]{Assumption}%
\newtheorem{remark}[theorem]{Remark}%
\newcommand{\sgn}{\mathrm{sgn}}

\newcommand{\fspace}[1]{{\mathsf{#1}}}
\newcommand{\fspaceL}{\fspace{L}}
\newcommand{\fspaceC}{\fspace{C}}
\newcommand{\fspaceW}{\fspace{W}}

\newcommand{\deq}{{:=}}

\newcommand{\ol}[1]{{\overline{#1}}}

\newcommand{\phase}{{\varphi}}

\newcommand{\Rset}{{\mathbb{R}}}

\newcommand{\Nset}{{\mathbb{N}}}

\newcommand{\ocinterval}[2]{(#1,\,#2]}%
\newcommand{\cointerval}[2]{[#1,\,#2)}%
\newcommand{\oointerval}[2]{(#1,\,#2)}%
\newcommand{\ccinterval}[2]{[#1,\,#2]}%

\newcommand{\R}{\mathbb{R}}

\newcommand{\Z}{\mathbb{Z}}

\newcommand{\rmd}{\mathrm{d}}
\newcommand{\rme}{\mathrm{e}}

\newcommand{\vph}{\varphi}

\newcommand{\sh}{{\rm sh}}

\newlength{\mhpicDwidth}
\newlength{\mhpicDvsep}
\newlength{\mhpicDhsep}

\newlength{\mhpicPwidth}
\newlength{\mhpicPvsep}
\newlength{\mhpicPhsep}

\newlength{\mhpicWhsep}

\newcommand{\pair}[2]{{\left({#1},\,{#2}\right)}}
\newcommand{\skp}[2]{{\left\langle{#1},\,{#2}\right\rangle}}

\newcommand{\at}[1]{{\left({#1}\right)}}
\newcommand{\nat}[1]{(#1)}
\newcommand{\bat}[1]{{\big(#1\big)}}
\newcommand{\Bat}[1]{{\Big(#1\Big)}}

\newcommand{\ato}[1]{{\left[{#1}\right]}}

\newcommand{\bato}[1]{{\big[#1\big]}}

\newcommand{\ul}[1]{\underline{#1}}

\newcommand{\bigpar}{\par\quad\newline\noindent}

\newcommand{\wt}[1]{{\widetilde{#1}}}
\newcommand{\wh}[1]{{\widehat{#1}}}

\newcommand{\jump}[1]{{|\![#1]\!|}}
\newcommand{\mean}[1]{{\langle#1\rangle}}

\newcommand{\norm}[1]{\parallel\!{#1}\!\parallel}
\newcommand{\abs}[1]{\left|{#1}\right|}
\newcommand{\babs}[1]{\big|{#1}\big|}

\newcommand{\dint}[1]{\,\mathrm{d}#1}

\newcommand{\ga}{{\gamma}}

\newcommand{\eps}{{\varepsilon}}

\newcommand{\la}{{\lambda}}
\newcommand{\si}{{\sigma}}

\newcommand{\MaFld}[1]{\overline{#1}}

\newcommand{\MiTime}{t}
\newcommand{\MiSpace}{x}

\newcommand{\MiLagr}{\alpha}

\newcommand{\MaSpace}{\MaFld{\MiSpace}}

\newcommand{\MaTime}{\MaFld{\MiTime}}
\newcommand{\MaLagr}{\MaFld{\MiLagr}}

\newcommand{\MaLagrDer}[1]{\partial_{\,\MaLagr\,}{#1}}

\newcommand{\MaTimeDer}[1]{\partial_{\,\MaTime\,}{#1}}
\newcommand{\calA}{\mathcal{A}}

\newcommand{\calC}{\mathcal{C}}

\newcommand{\calH}{\mathcal{H}}

\newcommand{\calJ}{\mathcal{J}}

\newcommand{\calL}{\mathcal{L}}

\newcommand{\calO}{\mathcal{O}}

\newcommand{\calT}{\mathcal{T}}

\usepackage{float}%
\begin{document}%
%
%
\title{Heteroclinic travelling waves in convex FPU-type chains}%
\date{\today}%
\author{Michael Herrmann\thanks{ %
    University of Oxford,
    Mathematical Institute, Centre for Nonlinear PDE (OxPDE),
    24-29 St Giles', Oxford OX1 3LB, United Kingdom, michael.herrmann@maths.ox.ac.uk.}
\and
Jens D.M. Rademacher\thanks{%
    National Research Centre for Mathematics and Computer Science (CWI),
    Science Park 123, 1098 XG Amsterdam, the Netherlands,  rademach@cwi.nl.}
}
\maketitle
%
%
%
\begin{abstract}%
We consider infinite FPU-type atomic chains with general convex
potentials and study the existence of monotone fronts that are
heteroclinic travelling waves connecting constant asymptotic states.
Iooss showed that small amplitude fronts bifurcate from
convex-concave turning points of the force. In this paper we prove
that fronts exist for any asymptotic states that satisfy
certain constraints. For potentials whose derivative has exactly one
turning point these constraints precisely mean that the front
corresponds to an energy conserving supersonic shock of the
`p-system', which is the naive hyperbolic continuum limit of the
chain. The proof goes via minimizing an action functional for the
deviation from this discontinuous shock profile. We also discuss
qualitative properties and the numerical computation of fronts.
\end{abstract}%
%
{Keywords:} %
\emph{Fermi-Pasta-Ulam chain}, %
\emph{heteroclinic travelling waves}, %
\emph{conservative shocks}, %
\emph{least action} %
\bigpar\noindent
MSC (2000): %
37K60, 
47J30, 
70F45, 
74J30 
%
%
%
%
\section{Introduction}
%
%
We consider infinite chains of identical particles as plotted in
Figure~\ref{IntroFig1}. These are nearest neighbour coupled in a
convex potential $\Phi:\R\to \R$ by Newton's equations
\begin{equation}
\label{e:FPU}%
\ddot x_\alpha = \Phi'(x_{\alpha+1} - x_\alpha) - \Phi'(x_\alpha -
x_{\alpha-1}),
\end{equation}
where $\dot{~} = \frac{\rmd}{\rmd t}$ is the time derivative,
$x_\alpha(t)$ the atomic position, $\alpha\in\Z$ the particle index.
\begin{figure}[ht!]%
  \centering{%
  \includegraphics[width=0.7\textwidth,draft=\figdraft]%
  {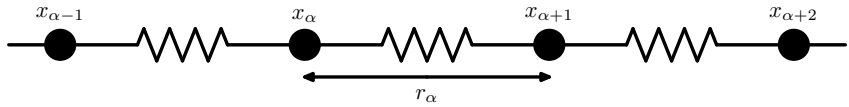}%
  }%
  \caption{The atomic chain with nearest neighbour interaction.}%
  \label{IntroFig1}%
\end{figure}%
\par
Such chains model particles connected by springs in one dimension
and serve as simplified models for crystals and solids. In their 
seminal paper \cite{FPU55} Fermi, Pasta and Ulam studied such chains
assuming that the interaction potential $\Phi$ contains at most quartic terms. 
We consider convex $\Phi$ with nonlinear force
function $\Phi^\prime$ and allow for turning points, i.e., points
where $\Phi^{\prime\prime\prime}=0$, but still refer to
\eqref{e:FPU} as FPU chains. In fact, the existence of fronts, which
is studied in this paper, requires that $\Phi^\prime$ has at least
one turning point, see \cite{HR08a}, and this excludes, for
instance, the famous Toda potential.
\par
One focus of the research concerns solutions with coherent
structures, in particular travelling waves, for which there exists a
smooth profile that travels with constant speed and shape through
the chain. The main types of travelling waves are periodic
\emph{wave trains}, homoclinic \emph{solitons} (or \emph{solitary
waves}), and heteroclinic \emph{fronts}.  Mathematically, rigorous
existence proofs of such solutions are an important basic issue.
During the last two decades a lot of research addressed the
existence of solitons and wave trains: \cite{FW94,FV99,Her05,DHM06}
establish the existence of such waves by solving constrained
optimisation problems, \cite{SW97,PP00,Pan05,SZ07} apply the
Mountain Path Theorem to the action integral for travelling waves,
and \cite{Ioo00} uses center manifold reduction with respect to the
spatial dynamics.
\par
In comparison, little is known rigorously for fronts. For
(non-smooth and non-convex) double-well potentials composed of the
same quadratic parabolas, the existence of fronts connecting
oscillatory states has been recently shown in \cite{SZ08}. Such
fronts can be interpreted as phase transitions and more physical
results can be found for instance in \cite{SCC05,TV05}.
\par
In these cases the connection between fronts and shocks in the naive
continuum limit of \eqref{e:FPU} was crucial. This limit is the
so-called p-system formed by the hyperbolic conservation laws for
mass and momentum. Shocks in the p-system come in different types
given by the relation of their speed and the sound speed of the
asymptotic states: shocks that are faster (slower) than these sound
speeds are called supersonic (subsonic). The fronts found in
\cite{SZ08} for the quadratic double-well case correspond to
subsonic shocks when taking the average of the asymptotic
oscillations. It turns out that these shocks also balance the energy
jump relation, and we refer to such shocks as conservative.
\bigpar
The connection to p-system shocks is also crucial for our result,
but we solely consider \emph{convex} potentials and the fronts we
find are \emph{supersonic}, \emph{conservative}, and \emph{monotone}
with constant asymptotic states, see Figure~\ref{IntroFig2}.  For
such potentials, the only previous result concerning fronts we are
aware of is the bifurcation result by Iooss in \cite{Ioo00} for
supersonic fronts of small amplitude connecting constant states near
a convex-concave turning point of $\Phi^\prime$. A physical
interpretation of these solutions is a deformation of the solid or
crystal, though the fronts can travel in either direction for the
same asymptotic states. Related `pulsating' travelling waves have
been studied in \cite{IJ05}. Here the profile of the wave changes
periodically while moving and the asymptotic states are wave trains.
It was shown that in the truncated normal form on the centre
manifold, pulsating front solutions bifurcate from convex-concave
turning points of $\Phi^\prime$.
\par
The analytical investigations in this paper are motivated by the
numerical simulations of atomistic Riemann problems which have
recently been studied by the authors in \cite{HR08a}. We observed
fronts in numerical simulations of \eqref{e:FPU} even for initial
data that are far from the data of a front, see
Figure~\ref{IntroFig2}. Hence, fronts are dynamically stable and
provide fundamental building blocks for atomistic Riemann solvers.
\begin{figure}[ht!]
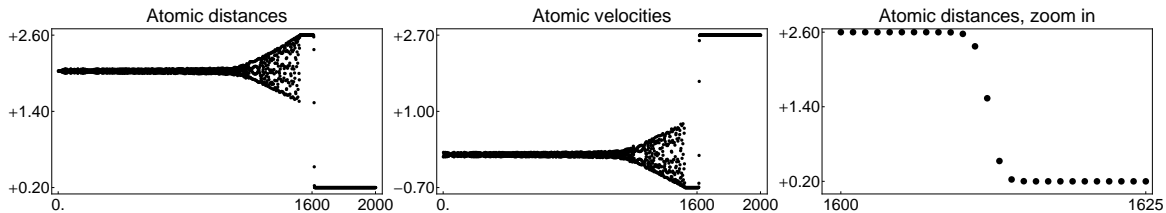
%
  \centering{%
  \includegraphics[width=0.32\textwidth, draft=\figdraft]%
  {\figfile{dist_ma}}%
  \includegraphics[width=0.32\textwidth, draft=\figdraft]%
  {\figfile{vel_ma}}%
  \includegraphics[width=0.32\textwidth, draft=\figdraft]%
  {\figfile{dist_mi}}%
  }%
  \caption{Fronts appear in the numerical simulations of FPU chains:
    snapshots of the atomic data with monotone front
    around particle $\alpha=1613$.}%
  \label{IntroFig2}%
\end{figure}%
\bigpar
The main result of this paper is the following.
\begin{theorem}
\label{t:main}%
For all convex and smooth potentials $\Phi$ the following assertions
are satisfied:
\begin{enumerate}
\item
Each front in the chain corresponds to a conservative shock in the
p-system.
\item
For each supersonic conservative shock in the p-system that
satisfies a certain area condition there exists a corresponding
monotone front in the chain.
\end{enumerate}
Here `correspond' means that shock and front have the same wave
speed and asymptotic states.
\end{theorem}
The precise assertions of this theorem depend on the set of turning
points of $\Phi^\prime$:
\begin{enumerate}
\item
If $\Phi^\prime$ has no turning points, then all shocks are
non-conservative and fronts do not exist.
\item
For strictly convex-concave $\Phi^\prime$ all conservative shocks
are supersonic and satisfy the area condition. In this case Theorem
\ref{t:main} implies that \emph{each} conservative shock can be
realised by an atomistic front. In particular, this allows for
arbitrary large jumps between the asymptotic states and extends 
the bifurcation result from \cite{Ioo00}.
\item
For strictly concave-convex $\Phi^\prime$ all conservative shocks
are subsonic and Theorem \ref{t:main} does not imply an existence
result for fronts.
\item
If $\Phi^\prime$ has more than one turning point, then Theorem
\ref{t:main} provides the existence of fronts for a proper subset of
all conservative shocks.
\end{enumerate}

We emphasise that Theorem \ref{t:main} does not imply the uniqueness
of fronts with prescribed asymptotic states. Moreover, not every
conservative shocks in the p-system corresponds to a front in the
FPU chain. For functions $\Phi^\prime$ with one turning point the
simulations in \cite{HR08a} rather suggest that only supersonic
shocks allow for an atomistic realization as a (dynamically stable)
front. This is supported by the bifurcation result in \cite{Ioo00},
which in fact disproves the existence of subsonic conservative
shocks with small jump heights, and our results, which imply that
subsonic shocks cannot minimize the action, compare the discussion
at the end of \S\ref{s:main}.
\par
We also mention the results about atomistic fronts by Aubry and
Proville. In \cite{AP07} they study FPU chains with concave
$\Phi^\prime$ -- so all shocks in the p-system are non-conservative
-- and allow for atomistic fronts by adding some damping force that
respects uniform motion. By numerical simulations they found
oscillations in the tails of such fronts and observed that these
oscillations persist when the damping tends to zero. We therefore
expect that non-conservative shock data can be related to
heteroclinic connections of wave trains.
\bigpar
The proof of the first part of Theorem~\ref{t:main} is fairly
straightforward and contained in Lemma~\ref{Lemma:JumpCond} below.
Proving the second part, which is mainly contained in
Theorem~\ref{EM.ExistenceTheo}, requires a number of preparatory
steps, and uses the convexity of $\Phi$, the supersonic front speed,
and the area condition in various fundamental steps. To guide the
reader we give an overview of the key ideas and steps for the proof.
\begin{enumerate}
\item
We use the a priori knowledge of the front speed in order to
reformulate the problem as a nonlinear fixed point equation for a
suitably normalised profile.
\item
We identify an action functional for the deviation from the
discontinuous shock profile such that the fixed point equation is
the corresponding Euler-Lagrange equation.
\item
We use the invariance of the set of monotone profiles
under the gradient flow of the Lagrangian to connect
stationary points in this set with fronts.
\item
We establish bounds for the action functional and use the direct
approach to show that the Lagrangian attains its minimum in the
set of monotone profiles.
\end{enumerate}
This paper is organised as follows. In \S\ref{s:set} we give a more
formal background for fronts and shocks of the p-system, and
\S\ref{s:main} contains the proof of the existence result. In
\S\ref{s:qualiter} we prove that these fronts decay exponentially to their
asymptotic states; the proof also applies to supersonic solitons
with monotone tails and extends the known results there. In addition, we discuss numerical computations
based on the gradient flow, and hereby illustrate the role of area
condition, supersonic front speed, and energy conservation.
%
\section{Fronts, shocks and normalisation}\label{s:set}
%
%
For our purposes it is convenient to use the atomic distances
$r_\MiLagr=x_{\MiLagr+1}-x_\MiLagr$ and velocities
$v_\MiLagr=\dot{x}_\MiLagr$ as the basic variables, changing
(\ref{e:FPU}) to the system
\begin{align*}%
\dot{r}_\MiLagr = v_{\MiLagr+1}-v_\MiLagr\;,\qquad \dot{v}_\MiLagr =
\Phi^{\prime}\at{r_\MiLagr} - \Phi^{\prime}\at{r_{\MiLagr-1}}.
\end{align*}
Travelling waves are exact solutions to the infinite chain
(\ref{e:FPU}) that depend on a single phase variable
$\phase=\alpha-\si{t}$, where the \emph{speed} $\si$ equals the
phase velocity. In terms of atomic distances and velocities
travelling waves can be written as
\begin{align*}
r_\MiLagr\at\MiTime=R\at{\phase},\qquad
v_\MiLagr\at\MiTime=V\at{\phase},
\end{align*}
where the profile functions $R$ and $V$
solve the advance-delay differential equations
\begin{align}%
\label{e:tw} %
\si\tfrac{\dint}{\dint\phase}R(\phase)+
V(\phase+1)-V(\phase)%
=0,\qquad
\si\tfrac{\dint}{\dint\phase}V(\phase)+
\Phi^\prime\bat{R\at{\phase}}-
\Phi^\prime\bat{R\at{\phase-1}}=0,
\end{align}
which imply the energy law
\begin{align}%
\label{e:tw.energy} %
\si\tfrac{\dint}{\dint\phase}
\Bat{\tfrac{1}{2}V^2\at\phase+\Phi\at{R\at\phase}}+
\Phi^\prime\bat{R\at\phase}V\at{\phase+1}-
\Phi^\prime\bat{R\at{\phase-1}}V\at{\phase}=0.
\end{align}
The travelling waves in our context are \emph{fronts}, which connect
two different constant asymptotic states for $\phase\to-\infty$ and
$\phase\to\infty$, i.e.,
\begin{align}%
\label{e:AsymptoticStates} %
R,\,V\in\fspaceC^1\cap\fspaceL^\infty
\quad\text{with}\quad%
R\at\phase\xrightarrow{\phase\to\pm\infty}r_\pm
\quad\text{and}\quad%
V\at\phase\xrightarrow{\phase\to\pm\infty}v_\pm,
\end{align}%
where $\fspaceL^p$ and $\fspace{C}^k$ denote the usual function
spaces on the real line.
The other variants of travelling waves mentioned in the introduction
are \emph{wave trains}, with periodic $R$ and $V$, and
\emph{solitons}, that are localised over a constant background
state, i.e., $r_+=r_-$ and $v_+=v_-$.
\par
In~\cite{Ioo00} it has been proven that fronts bifurcate from
turning points $r_*$ of $\Phi^\prime$ with $\Phi^{(4)}(r_*)\neq 0$
if and only if $\Phi^{(4)}(r_*)<0$. These fronts have strictly
monotone profiles that converge exponentially to their asymptotic
states. Our main existence result extends this in the sense that we
prove existence of fronts for any left and right states that satisfy
certain constraints and can have large amplitude.
\bigpar
As a first result we establish necessary conditions for the
existence of fronts, which reflect that each front converges to a
shock in a large-scale limit. These conditions turn out to be the
jump conditions for an energy conserving shock of the p-system (see
below), and show that the speed $\si$ of a front is completely
determined by its asymptotic states.
\par
For any atomic observable
$\psi=\psi\pair{r}{v}$ we
define the mean and the jump by
\begin{align*}
\mean\psi:=\tfrac{1}{2}\bat{\psi\pair{r_+}{v_+}+\psi\pair{r_-}{v_-}}
\quad\text{and}\quad%
\jump{\psi}:=\psi\pair{r_+}{v_+}-\psi\pair{r_-}{v_-},
\end{align*}
respectively, where $r_\pm$ and $v_\pm$ denote the asymptotic values
of a front, see \eqref{e:AsymptoticStates}.
\begin{lemma}%
\label{Lemma:JumpCond}
Suppose that $R$ and $V$ as in \eqref{e:AsymptoticStates} solve the
front equation \eqref{e:tw}. Then,
\begin{align}
\label{Theo:JumpCond.Eqn1}
\si\jump{r}+\jump{v}=0,\qquad
\si\jump{v}+\jump{\Phi^\prime\at{r}}=0,\qquad
\si\jump{\tfrac{1}{2}v^2+\Phi\at{r}}+\jump{\Phi^\prime\at{r}v}=0.
\end{align}
\end{lemma}%
\begin{proof}%
Integrate \eqref{e:tw}$_{1}$ over a finite interval
$\ccinterval{-n}{n}$ gives
\begin{align*}
\si\at{R\at{n}-R\at{-n}}=
-\int\limits_{-n+1}^{n+1}{V\at{\phase}}\dint\phase
+\int\limits_{-n}^{n}{V\at{\phase}}\dint\phase
=
\int\limits_{-n}^{-n+1}{V\at{\phase}}\dint\phase-
\int\limits_{n}^{n+1}{V\at{\phase}}\dint\phase,
\end{align*}
and the limit $n\to\infty$ yields \eqref{Theo:JumpCond.Eqn1}$_1$.
The jump conditions \eqref{Theo:JumpCond.Eqn1}$_2$ and
\eqref{Theo:JumpCond.Eqn1}$_3$ follow similarly from
\eqref{e:tw}$_{2}$ and \eqref{e:tw.energy}.
\end{proof}%
We mention that the jump conditions \eqref{Theo:JumpCond.Eqn1} are
also derived in \cite{AP07} and that a similar method is used in
\cite{Ser07} to derive jump conditions from purely discrete front
equations.
%
%
%
\subsection{The p-system and conservative shocks}
%
%
The p-system can be formally derived from the FPU chain as a
continuum limit in the \emph{hyperbolic scaling} of the
\emph{microscopic} coordinates $\MiTime$ and $\MiLagr$. This scaling
bridges to the \emph{macroscopic} time $\MaTime=\varepsilon\MiTime$
and particle index $\MaLagr=\varepsilon\MiLagr$, where $0<\eps\ll1$
is the scaling parameter. It is natural to scale the atomic
positions in the same way, i.e. $\MaSpace=\varepsilon\MiSpace$,
which leaves atomic distances and velocities scale invariant.
\par
Substituting an ansatz of macroscopic fields $r(\MaTime,\MaLagr)$
and $v(\MaTime,\MaLagr)$ such that $r_\MiLagr\at\MiTime=
r\pair{\varepsilon\MiTime}{\varepsilon\MiLagr}$,
\begin{math}
v_\MiLagr\at\MiTime=%
v\pair{\varepsilon\MiTime}{\varepsilon\MiLagr}
\end{math} %
into \eqref{e:FPU} and taking the limit $\varepsilon\rightarrow0$
yields the macroscopic conservation laws for mass and momentum
\begin{align}%
\label{e:ColdModEqn} %
\MaTimeDer{r}-%
\MaLagrDer{v}=0,\qquad
\MaTimeDer{v}-%
\MaLagrDer{\Phi^\prime\at{r}}%
&=0,%
\end{align}
which is the aforementioned p-system. It is well known that the
p-system is hyperbolic if and only if $\Phi$ is convex, and that for
smooth solutions the energy is conserved via
\begin{align}
\notag
\MaTimeDer{}\at{\textstyle{\frac{1}{2}v^2}+\Phi\at{r}}-%
\MaLagrDer{}\at{v\,\Phi^\prime\at{r}}=0.%
\end{align}
Next we summarise some facts about shocks in the p-system. For more
details we refer to standard textbooks \cite{Smo94,Daf00,LeF02}, and
to \cite{HR08a} for conservative shocks.
\par
A shock connecting a left state $u_-=(r_-,v_-)$ to a right state
$u_+=(r_+,v_+)$ propagates with a constant shock speed $\si$ so that
$u_-$ and $u_+$ satisfy the Rankine-Hugeniot jump conditions for
mass and momentum \eqref{Theo:JumpCond.Eqn1}$_1$ and
\eqref{Theo:JumpCond.Eqn1}$_2$. Corresponding to the two
characteristic velocities of \eqref{e:ColdModEqn}  one distinguishes
between $1$-shocks with $\si<0$ and $2$-shocks with $\si>0$.
\par
An important observation is that for either convex or concave flux
$\Phi^\prime$ the jump conditions \eqref{Theo:JumpCond.Eqn1}$_1$ and
\eqref{Theo:JumpCond.Eqn1}$_2$ imply that the jump condition for the
energy \eqref{Theo:JumpCond.Eqn1}$_3$ must be violated, that means
\begin{align}%
\label{e:ColdDataEnJump}%
\si\jump{\tfrac{1}{2}v^2+\Phi\at{r}}+%
\jump{v\,\Phi^\prime\at{r}}\neq0.
\end{align}
For non-convex or non-concave $\Phi^\prime$, however, it is possible
to find \emph{conservative shocks} that satisfy all three jump
conditions from \eqref{Theo:JumpCond.Eqn1}, see Theorem 5.1 in
\cite{HR08a}. Eliminating $\jump{v}$ and $\sigma$ in
\eqref{Theo:JumpCond.Eqn1} shows that each conservative shock
satisfies
\begin{equation}\label{e:ConsJump}
\calJ(r_-,r_+):= \jump{\Phi(r)} - \jump{r}\langle{\Phi'(r)}\rangle
=0.
\end{equation}
Conversely, for each solution to \eqref{e:ConsJump} there exist, up
to Galilean transformations, exactly two corresponding conservative
shocks which differ only in $\sgn\jump{v}=-\sgn\,{\si}$ and
propagate in opposite directions. The geometric interpretation of
\eqref{e:ConsJump} is that the signed area between the graphs of
$\Phi'$ and the secant line through $\Phi'(r_-)$ and
$\Phi^\prime(r_+)$ is zero over $[r_-,r_+]$, compare
Figure~\ref{fig:forces} below.
\par
For the harmonic potential $\Phi\at{r}=\tfrac{c}{2}{r^2}$ the
p-system is linearly degenerate, hence all shocks are conservative.
In the nonlinear case, however, the solution set to
\eqref{e:ConsJump} is either empty or the union of curves in the
$(r_-,r_+)$-plane, see Figure~\ref{fig:consShock} for an example. In
particular, precisely from turning points of $\Phi'$ there bifurcate
curves of conservative shocks, see \cite{HR08a}.
\begin{figure}[t!]%
\centering{ %
\setlength{\tabcolsep}{0cm}%
\begin{minipage}[c]{\mhpicDwidth}%
\includegraphics[width=\mhpicDwidth, draft=\figdraft]%
{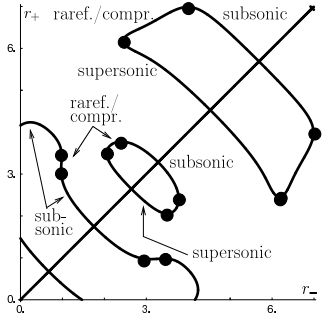}%
\end{minipage}%
}%
\caption{%
  Conservative shocks for the potential $\Phi(r+1) = r + \frac 1 2r^2
  + \frac 1 {20}r^3 - \frac 1 4 \cos(2r) + \frac 1 {10} \sin(3r )$:
  solid lines form the solution set to $\calJ(r_-,r_+)=0$, i.e., the
  distance data for conservative shocks; bullets lie at local extrema
  of the curves in the coordinate directions, and are the points where
  the shock type changes.} %
\label{fig:consShock}%
\end{figure}%
\bigpar%
Shocks come in different types determined by the relation of the
shock speed $\si$ and the sound speeds
$\lambda(r):=\pm\sqrt{\Phi''(r)}$ evaluated in $r_-$ and $r_+$,
where the signs of both the sound speed and the shock speed are
taken negative (positive) when considering 1-shocks (2-shocks). More
precisely, a shock connecting $u_-$ to $u_+$ with speed $\si$ is
called
\begin{enumerate}
\item
\emph{compressive}, or \emph{Lax shock}, if
$\lambda\at{r_{-}}>\si>\lambda\at{r_{+}}$,
\item
\emph{rarefaction shock}, if
$\lambda\at{r_{-}}<\si<\lambda\at{r_{+}}$,
\item
\emph{supersonic}, or \emph{fast undercompressive}, if
$\abs{\si}>\abs{\lambda\at{r_{-}}}$ and
$\abs{\si}>\abs{\lambda\at{r_{+}}}$,
\item
\emph{subsonic}, or \emph{slow undercompressive}, if
$\abs{\si}<\abs{\lambda\at{r_{-}}}$ and
$\abs{\si}<\abs{\lambda\at{r_{+}}}$,
\item
\emph{sonic}, if $\abs{\si}=\abs{\lambda\at{r_{-}}}$ or
$\abs{\si}=\abs{\lambda\at{r_{+}}}$.
\end{enumerate}
For all sufficiently small jump heights
$\abs{r_+-r_-}+\abs{v_+-v_-}$ one can show that compressive and
rarefaction shocks have negative and positive, respectively, energy
production in \eqref{e:ColdDataEnJump}. Consequently, all
conservative shocks with small jump height are either sonic or
undercompressive.  The conservative shocks bifurcating from turning
points $r_*$ of $\Phi^\prime$ are supersonic if $\Phi^{(4)}(r_*)<0$
and subsonic if $\Phi^{(4)}(r_*)>0$, see \cite{HR08a}. %
%
%
%
%
%
\subsection{Normalisation and reformulation}
%
%
Next we show that, after a suitable renormalisation, the front
equations for $R$ and $V$ are equivalent to a nonlinear fixed point
equation for a normalised profile $W$ with asymptotic states $\pm
1$. To this end we introduce
the averaging operator
\begin{align*}
\at{\calA{U}}\at\phase:=
\int\limits_{\phase-\tfrac{1}{2}}^{\phase +
\tfrac{1}{2}} U\at{\tilde\phase}\dint{\tilde\phase},
\end{align*}
which satisfies
\begin{align}
\label{Def.AvOp.Props}%
\tfrac{\dint}{\dint\phase}\at{\calA{U}}=
\nabla{U}:=%
U\at{\cdot+\tfrac{1}{2}}-U\at{\cdot-\tfrac{1}{2}},
\qquad%
\lim\limits_{\phase\to\pm\infty}\at{\calA{U}}\at\phase=
\lim\limits_{\phase\to\pm\infty}{U}\at\phase.
\end{align}
For given asymptotic front states $\pair{r_\pm}{v_\pm}$ we further
introduce a normalised potential $\wh{\Phi}$ by
\begin{align*}
\wh{\Phi}\at{w}=\frac{4}{\jump{\Phi^\prime\at{r}}\jump{r}}
{\Phi}\Bat{\mean{r}+\tfrac{1}{2}\jump{r}\,w}
-\frac{2\mean{\Phi^\prime\at{r}}}{\jump{\Phi^\prime\at{r}}}\,{w},
\end{align*}
which is convex on $\ccinterval{-1}{1}$ and strictly convex if
$\Phi$ is.  Moreover, the jump conditions from
\eqref{Theo:JumpCond.Eqn1} imply $\Phi^\prime\at{\pm1}=\pm1$ and
$\wh{\Phi}\at{1}=\wh{\Phi}\at{-1}$, and the shock corresponding to
$\pair{r_\pm}{v_\pm}$ is supersonic if and only if
$\wh{\Phi}^{\prime\prime}\at{1}<1$ and
$\wh{\Phi}^{\prime\prime}\at{-1}<1$.
\begin{lemma}
\label{Norm.Problem}%
For fixed asymptotic states $\pair{r_\pm}{v_\pm}$ there is a
one-to-one correspondence between front solutions to \eqref{e:tw}
and solutions $W$ to
\begin{align}
\label{Norm.Problem.Eqn}
W=\calA\wh{\Phi}^\prime\at{\calA{W}},
\end{align}
with $W\at\phase\xrightarrow{\phase\to\pm\infty}\pm1$.
\end{lemma}
\begin{proof}
Suppose that $R$ and $V$ solve \eqref{e:tw}. Then there exists a
normalised profile $W$ such that
\begin{align}
\label{Norm.Problem.Proof1}
V=\mean{v}+\tfrac{1}{2}\jump{v}{W}.
\end{align}
By \eqref{Def.AvOp.Props} the first equation from \eqref{e:tw} is
equivalent to
\begin{align*}
-\si\wh{R}=\calA{V}+c=\mean{v}+\tfrac{1}{2}\jump{v}\calA{W}+c
\end{align*}
with  $\wh{R}\at{\phase}=R\at{\phase-\tfrac{1}{2}}$ and an
integration constant $c$. Passing to $\phase\to-\infty$ and
$\phase\to\infty$ we find
\begin{align}
\notag
-\si\jump{r}=\jump{v},\quad{c}=-\si\mean{r}-\mean{v},
\end{align}
and this gives
\begin{align}
\label{Norm.Problem.Proof2}
\wh{R}=\mean{r}+\tfrac{1}{2}\jump{r}{\calA{W}}.
\end{align}
The second equation from \eqref{e:tw} transforms into
\begin{align}
\label{Norm.Problem.Proof4}
-\si{V}=\calA{\Phi^\prime}\nat{\wh{R}}+d,
\end{align}
and inspecting the asymptotic values we find
\begin{align}
\label{Norm.Problem.Proof5}
-\si\jump{v}=\jump{\Phi^\prime\at{r}},\quad{d}
=%
-\si\mean{v}-\mean{\Phi^\prime\at{r}}.
\end{align}
Inserting \eqref{Norm.Problem.Proof1} and
\eqref{Norm.Problem.Proof2} into \eqref{Norm.Problem.Proof4} gives
\begin{align*}
W=
-\frac{2}{\si\jump{v}}
\calA\Bat{{\Phi^\prime\at{\mean{r}+
\tfrac{1}{2}\jump{r}\calA{W}}}-\mean{\Phi^\prime\at{r}}},
\end{align*}
which is \eqref{Norm.Problem.Eqn} due to
\eqref{Norm.Problem.Proof5}. Finally, it is easy to see that each
solution to \eqref{Norm.Problem.Eqn} determines a front via
\eqref{Norm.Problem.Proof1} and \eqref{Norm.Problem.Proof2}.
\end{proof}
The equivalence between the lattice equation \eqref{e:tw} for
travelling waves and the nonlinear eigenvalue problem
$\si^2{W}=\calA\wh{\Phi}^{\prime}\at{\calA{W}}$ is well established
for wave trains and solitons, compare for instance \cite{FW94,FV99},
but in the case of fronts we have not seen this formulation before.
An important difference is that for fronts the speed $\si$ is
completely determined by the asymptotic states and can hence be
removed from the problem.
\bigpar
We conclude with some remarks concerning the regularity of fronts.
Suppose that $W\in\fspaceL^\infty$ solves \eqref{Norm.Problem.Eqn}.
Using \eqref{Def.AvOp.Props}, compare also Lemma \ref{Props.AvOp}
below, and the smoothness of $\Phi$ we find
\begin{align}\label{e:wpp}
\tfrac{\dint^2}{\dint\phase^2}{W}=
\nabla\bato{\wh{\Phi}^{\prime\prime}\at{\calA{W}}\nabla{W}},
\end{align}
and hence $W\in\fspace{W}^{2,\infty}\subset\fspaceC$. Moreover, we
infer that $W\in\fspaceC^2$ and, differentiating \eqref{e:wpp}
further, that $W$ is as smooth as $\Phi$. Finally, exploiting
\eqref{Norm.Problem.Proof1} and \eqref{Norm.Problem.Proof2} we
arrive at the following result and refer to \S\ref{s:qual} for
further qualitative properties of fronts.
\begin{remark}
\label{Rem:Regularity} Each solution $W\in\fspaceL^\infty$ to 
\eqref{Norm.Problem.Eqn} is contained in $\fspace{C}^2$ and provides
a smooth solution to \eqref{e:tw} with at least
$R,V\in\fspace{C}^2\cap\fspaceL^\infty$.
\end{remark}
%
%
%
%
%
\section{Proof of the existence result}\label{s:main}
%
\newcommand{\cond}[1]{(#1)}
In order to prove the existence of monotone solutions $W$ to the
fixed point problem \eqref{Norm.Problem.Eqn} we consider only
normalised potentials $\Phi=\wh{\Phi}$ and rely on the following
standing assumption. This involves the function
\begin{align}
\label{Eqn:DefG}
g_\Phi\at{w}
\deq%
\int\limits_{-1}^{w}v-\Phi^\prime\at{v}\dint{v}
=%
\Phi\at{-1}-\Phi\at{w}+\tfrac{1}{2}w^2-\tfrac{1}{2},
\end{align}
which measures the signed area between the identity and the graph of
$\Phi^\prime$, see Figure \ref{fig:forces}. 
\begin{assumption}
\quad \label{MainAss}
The potential $\Phi$ satisfies
\begin{enumerate}
\item[\cond{R}]%
regularity: $\Phi$ is twice continuously differentiable on
$\ccinterval{-1}{1}$,
\item[\cond{N}]%
normalisation: $\Phi^\prime\at{-1}=-1$ and
$\Phi^\prime\at{1}=1$,
\item[\cond{C}]%
convexity: $\Phi^{\prime\prime}$ is nonnegative on
$\ccinterval{-1}{1}$,
\item[\cond{E}]%
conservation of energy: $\Phi\at{-1}=\Phi\at{1}$,
\item[\cond{S}]%
supersonic front speed: $\Phi^{\prime\prime}\at{-1}<1$ and
$\Phi^{\prime\prime}\at{1}<1$.
\item[\cond{A}]%
positive signed area:
$g_\Phi\at{w}>0$ for all $w\in\oointerval{-1}{1}$.
\end{enumerate}
\end{assumption}
We mention that \cond{N} and \cond{R} are made for convenience,
\cond{E} is necessary according to Lemma \ref{Lemma:JumpCond}, and
\cond{C} is needed for the monotonicity of fronts.  The status of
\cond{S} and \cond{A} is more delicate: If $\Phi^\prime$ has only
one turning point, then \cond{S} coincides with the bifurcation
condition in \cite{Ioo00} and implies \cond{A}, see Remark
\ref{MainAss.Rem} below. If $\Phi^\prime$ has more than one turning
point, then \cond{A} is independent of \cond{S}. In any case both
\cond{S} and \cond{A} are closely related to the existence of
action-minimising fronts as discussed in more detail at
the end of \S\ref{s:main} and illustrated in \S\ref{s:num}.
\begin{figure}[ht!]
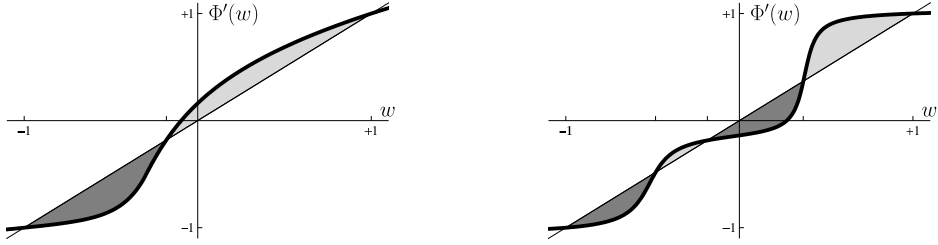
%
\centering{%
\includegraphics%
[width=0.3\textwidth, draft=\figdraft]%
{\figfile{forces_1}}%
\hspace{0.15\textwidth}%
\includegraphics%
[width=0.3\textwidth, draft=\figdraft]%
{\figfile{forces_2}}%
}%
\caption{%
  Two examples for force functions $\Phi^\prime$ that satisfy
  Assumption \ref{MainAss}; the first one is
  prototypical for the supersonic fronts with small amplitudes
  that bifurcate from convex-concave turning points.
  Condition \cond{A} precisely means that for each $-1<w<1$
  the signed area between the identity and the graph of $\Phi^\prime$ is
  positive in the stripe $\ccinterval{-1}{w}$;  the
  signed area vanishes in the stripe $\ccinterval{-1}{1}$
  if and only if the corresponding p-system shock conserves
  the energy via \cond{E}.
}%
\label{fig:forces}%
\end{figure}%
\bigpar 
Recall that turning points of $\Phi^\prime$ are exactly
those points in which $\Phi^{\prime\prime}$ changes its
monotonicity. More precisely, convex-concave (concave-convex)
turning points correspond to local maximisers (minimiser) of
$\Phi^{\prime\prime}$. We now summarize the main observations
concerning the turning points of $\Phi^{\prime}$.
\begin{remark}
\label{MainAss.Rem} %
Suppose that \cond{N}, \cond{R}, and $\cond{S}$ are satisfied. Then
$\Phi^\prime$ has at least one convex-concave turning point in
$\oointerval{-1}{1}$. Moreover, if $\Phi^\prime$ has no further
turning point in $\oointerval{-1}{1}$ then \cond{E} implies
\cond{A}.
\end{remark}
\begin{proof}
Suppose for contradiction that
$\max{\,}\Phi^{\prime\prime}|_{\ccinterval{-1}{1}}\leq1$. Then, for
all $w\in\ccinterval{-1}{1}$ we find
$g_\Phi^{\prime\prime}\at{w}=1-\Phi^{\prime\prime}\at{w}\geq0$ and
hence
\begin{math}
0=g_\Phi^\prime\at{-1}
\leq
{g_\Phi}^\prime\at{w}\leq{g_\Phi}^\prime\at{+1}=0.
\end{math} %
This implies $g_\Phi^{\prime}\at{w}=0$ for all
$w\in\ccinterval{-1}{1}$, and hence a contradiction to \cond{S}. Therefore,
 $\Phi^{\prime\prime}$ possesses a local maximiser
$w\in\oointerval{-1}{1}$, which is a convex-concave turning point of
$\Phi^\prime$.
\par
Towards \cond{A} suppose additionally \cond{E} and that
$\Phi^\prime$ is convex-concave in $\ccinterval{-1}{1}$. Then,
$g_\Phi\at{1}=g_\Phi\at{-1}=0$, and the monotonicity properties of
$\Phi^{\prime\prime}$ imply that there exist
$-1\leq{w_1}\leq{w_2}\leq{1}$ such that
$\Phi^{\prime\prime}\at{w}\leq{1}$ and
$\Phi^{\prime\prime}\at{w}\geq{1}$ for all
$w\in{I_1}=\ccinterval{-1}{w_1}\cup\ccinterval{w_2}{1}$ and
$w\in{I_2}=\ccinterval{w_1}{w_2}£$, respectively. In particular,
$g_{\Phi}$ is convex in $I_1$ but concave in $I_2$, and combining
this with $0=g_\Phi\at{\pm1}=g_\Phi^\prime\at{\pm1}$
and $g_\Phi^{\prime\prime}\at{\pm1}>0$,
we readily find $g_\Phi\at{w}>0$
for all $w\in\oointerval{-1}{+1}$.
\end{proof}
%
%
%
\subsection{The action functional}
%
%
In this section we cast the problem in a variational framework and
show that
\begin{align*}
W=\calT\ato{W},\qquad\calT\ato{W} := \calA\Phi^\prime\at{\calA{W}},
\end{align*}
is the Euler-Lagrange equation for an action integral. To this
end we introduce the affine Banach space
\begin{align*}
\calH:=\Big\{%
\text{functions $W:\Rset\to\Rset$ such that $\wt{W}=W_\sh-W\in\fspaceL^2$}
\Big\},
\end{align*}
where the reference profile
\begin{math}
W_\sh\at{\phase}=\sgn\at\phase
\end{math}
is the shock profile that connects $-1$ to $1$. We also define the
Lagrangian on $\calH$ by
\begin{align}
\label{Def.Lagr}
\calL\at{W}=\int\limits_\Rset
L\pair{W}{\calA{W}}-L\pair{W_\sh}{\calA{W_\sh}}\dint\phase,\quad
L\pair{v}{r}=\tfrac{1}{2}v^2-\Phi\at{r},
\end{align}
and a formal calculation shows
$\partial\calL\ato{W}=W-\calT\ato{W}$. Notice that the ansatz
$\wt{W}\in\fspaceL^2$ arises naturally when computing the asymptotic
behaviour for $\phase\to\pm\infty$. In fact, since $\calA{\wt{W}}$
decays like $\wt{W}$ and due to $\Phi^\prime\at{\pm1}=-1$ the
integrand in \eqref{Def.Lagr} behaves like
$\frac{1}{2}\at{1-\Phi^{\prime\prime}\at{\pm1}}\wt{W}\at{\phase}^2$.

\bigpar%
In the remainder of this section we show that both $\calL$ and
$\calT$ are well-defined on $\calH$. We start with some properties
of the averaging operator $\calA$ that are proven in \cite{Her09}.
\begin{lemma}
\label{Props.AvOp}
The linear operator $\calA$ is well-defined on $\fspaceL^p$,
$1\leq{p}\leq\infty$, and has the following properties:
\begin{enumerate}
\item
${\calA}$ maps into $\fspaceW^{1,\,p}$ with
\begin{math}
\at{{\calA}{{W}}}^\prime\at\phase=
{W}\at{\phase+1/2}-{W}\at{\phase-1/2},
\end{math}
\item
${\calA}$ maps $\fspaceL^2$ into $\fspaceL^2\cap{}\fspaceL^\infty$
with
\begin{math}
\norm{{\calA}{\wt{W}}}_\infty\leq\norm{{\wt{W}}}_2
\end{math}
and
\begin{math}
\norm{{\calA}{\wt{W}}}_2\leq\norm{{\wt{W}}}_2
\end{math},
\item
${\calA}$ is self-adjoint on $\fspaceL^2$,
\item
If a sequence $\nat{\wt{W}_n}_n$ converges weakly in $\fspaceL^2$ to
some limit $\wt{W}_\infty$, then $\nat{{\calA}{\wt{W}}_n}_n$
converges strongly and pointwise to ${\calA}{\wt{W}}_\infty$ in
$\fspaceL^2\nat{\ccinterval{-M}{M}}$ for each $M<\infty$.
\end{enumerate}
\end{lemma}%
Our next result shows $\calT\ato{W}-W\in\fspaceL^2$ for all
$W\in\calH$ and implies that $\calT:\calH\to\calH$ is well-defined.
\begin{lemma}
%
For any $W\in\calH$ we have
\begin{align}
\label{Props.L2Estimates.Eqn1}
W-\Phi^\prime\at{\calA{W}}\;\in\;\fspaceL^2,\qquad
{W}-\calA\Phi^\prime\at{\calA{W}}\;\in\;\fspaceL^2,
\end{align}
as well as
\begin{align}
\label{Props.L2Estimates.Eqn2}
\int\limits_\Rset{W}\mu-\Phi^\prime\at{\calA{W}}\calA\mu\dint\phase=
\int\limits_\Rset{W}\mu-\calA\Phi^\prime\at{\calA{W}}\mu\dint\phase
\end{align}
for all test functions $\mu\in\fspaceL^2$.
\end{lemma}
\begin{proof}
With
\begin{align*}
\babs{
\Phi^\prime\bat{\at{\calA{W}}\at\phase}-
\Phi^\prime\bat{\at{\calA{W}_\sh}\at\phase}}
\leq%
C\babs{\calA{\wt{W}}\at\phase},
\quad%
C=\sup\Big\{\abs{\Phi^{\prime\prime}\at{w}}\,:\,\abs{w}
\leq\max\{1,\,\norm{\calA{W}}_\infty\}\Big\}
\end{align*}
for all $\phase$, and
$\Phi^\prime\at{\calA{W}_\sh\at\phase}=W_\sh\at\phase$ for
$\abs{\phase}\geq\tfrac{1}{2}$, we find (with the usual cut-off
function $\chi$)
\begin{align*}
\babs{W-\Phi^\prime\at{\calA{W}}}
\leq%
\wt{W}+\chi_{\ccinterval{-\tfrac{1}{2}}{\tfrac{1}{2}}}%
\babs{W_\sh-\Phi^\prime\at{\calA{W}_\sh}}+C\babs{\calA{\wt{W}}}.
\end{align*}
This gives \eqref{Props.L2Estimates.Eqn1}$_1$ since
$\wt{W},\,\calA\wt{W}\in\fspaceL^2$, and the proof of
\eqref{Props.L2Estimates.Eqn1}$_2$ is similar. Now let $M>0$ be
arbitrary. Since $\calA$ is symmetric, see Lemma \ref{Props.AvOp},
we write
\begin{align}
\label{Props.L2Estimates.Eqn3}
\int\limits_{-M}^{M}\Phi^\prime\at{\calA{W}}
\calA\mu\dint\phase
&=%
\int\limits_{\Rset}\mu\calA\at{\chi_{\ccinterval{-M}{M}}
\Phi^\prime\at{\calA{W}}}\dint\phase
=
\int\limits_{-M}^M\mu\calA\Phi^\prime\at{\calA{W}}\dint\phase+E_M,
\end{align}
where the error term satisfies
\begin{align*}
\abs{E_M}\leq{C}
\at{%
\int\limits_{M-\tfrac{1}{2}}^{M+\tfrac{1}{2}}\abs{\mu}\dint\phase
+%
\int\limits_{-M-\tfrac{1}{2}}^{-M+\tfrac{1}{2}}
\abs{\mu}\dint\phase
},\qquad
C=\sup\big\{\abs{\Phi^\prime\at{w}}\,:
\,\abs{w}\leq\norm{\calA{W}}_\infty\big\}.
\end{align*}
Finally, we add $\int_{M}^{M}{W}\mu\dint\phase$ on both sides of
\eqref{Props.L2Estimates.Eqn3}, and taking the limit $M\to\infty$
yields \eqref{Props.L2Estimates.Eqn2}.
\end{proof}
We now prove that $\calL$ is well-defined and
G\^{a}teaux-differentiable on $\calH$.
\begin{remark}
\label{EM.NormPot}%
For each $w_0$ the pointwise normalised potential
\begin{align*}
\Phi_{w_0}\at{w}:=
\Phi\at{w_0+w}-\Phi^\prime\at{w_0}w-\Phi\at{w_0}
 \end{align*}
satisfies $0\leq\Phi_{w_0}\at{w}\leq\tfrac{c}{2}w^2$ with
$c=\sup\Phi^{\prime\prime}|_{\ccinterval{w_0}{w}}$.
\end{remark}
\begin{lemma}
\label{DerivativeOfL} %
The functional $\calL$ is well-defined on $\calH$. Moreover, it is
G\^{a}teaux differentiable with locally Lipschitz continuous
derivative $\partial\calL\ato{W}=W-\calT\ato{W}$.
\end{lemma}
\begin{proof}
By construction, we have
\begin{align*}
L\pair{W}{\calA{W}}-L\pair{W_\sh}{\calA{W_\sh}}
=%
\tfrac{1}{2}\wt{W}^2-
W_\sh\wt{W}+
\Phi^\prime\at{\calA{W}_\sh}\calA\wt{W}-
\Phi_{\calA{W}_\sh}\nat{\calA\wt{W}}
\end{align*}
with pointwise normalised potentials $\Phi_{\calA{W}_\sh}$ as in
Remark \ref{EM.NormPot}. Using \eqref{Props.L2Estimates.Eqn2} with
$W=W_\sh$ and $\mu=\wt{W}$ we find
\begin{align*}
\int\limits_\Rset{W}_\sh\wt{W}-
\Phi^\prime\at{\calA{W}_\sh}\calA\wt{W}\dint\phase=
\int\limits_\Rset\Bat{{W}_\sh-
\calA\Phi^\prime\at{\calA{W}_\sh}}\wt{W}\dint\phase,
\end{align*}
with well-defined integrals on both sides.
Using
$
C:= \sup\Phi^{\prime\prime}|_{[-1,1]} < \infty
$
we estimate
\begin{align*}
\int\limits_\Rset\tfrac{1}{2}\wt{W}^2\dint\phase<\infty,
\quad%
\int\limits_\Rset\Phi_{\calA{W}_\sh}\nat{\calA\wt{W}}\dint\phase
\leq{}%
C\norm{\calA\wt{W}}_2^2
\leq{}%
C\norm{\wt{W}}_2^2,
\end{align*}
and conclude that $\calL\at{W}$ is well-defined for all $W\in\calH$.
Moreover, \eqref{Props.L2Estimates.Eqn1} implies
$W-\calT\ato{W}\in\fspaceL^2$, and in view of
\eqref{Props.L2Estimates.Eqn2} we readily verify the formula for
$\partial\calL$ as well as the claimed Lipschitz property.
\end{proof}
%
%
\subsection{Monotone and heteroclinic functions }\label{s:cones}
%
%
In order to introduce a refined ansatz space for the front profile
$W$ we start with some preliminary remarks concerning the weak
formulation of monotonicity, asymptotic values, and
$\fspaceL^\infty$ bounds.
\par%
Let $\cointerval{\phase_0}{\infty}$ be some interval,
$U\in\fspaceL^\infty\at{\cointerval{\phase_0}{\infty}}$ be fixed,
and let $\mu\geq0$ denote an arbitrary smooth test function with
compact support in $\oointerval{\phase_0}{\infty}$. Then, the
function $U$ is said to
\begin{enumerate}
\item
be \emph{increasing}, if
\begin{math}
\int_{\Rset}\at{{U}\at{\phase+\bar\phase}-{U}\at{\phase}}
\mu\at{\phase}\dint\phase\geq0
\end{math}
for all shifts $\bar\phase\geq0$ and all $\mu$,
\item
be \emph{nonnegative}, if
\begin{math}
\int_{\Rset}{U}\at{\phase}\mu\at{\phase}\dint\phase\geq0
\end{math}
for all $\mu$,
\item
\emph{take values} in $\ccinterval{u_1}{u_2}$, if
\begin{math}
u_2\geq\int_{\Rset}{U}\at{\phase}\mu\at{\phase}\dint\phase\geq{u_1}
\end{math}
for all $\mu\geq0$ with
$\int_{\Rset}\mu\at{\phase}\dint\phase=1$,
\item
have the \emph{asymptotic value} $u_+$ for $\phase\to\infty$, if
\begin{math}
\int_\Rset{U}\at{\phase+\bar\phase} \mu\at{\phase}\dint\phase
\to{u}_+\int_\Rset\mu\at{\phase}\dint\phase
\end{math}
as $\bar\phase\to\infty$ for all $\mu$.
\end{enumerate}
Note that all these properties are preserved under weak convergence,
and that similar statements hold on $\ocinterval{\infty}{\phase_0}$.
\bigpar
We aim at establishing the existence of fronts in the convex set
\begin{align*}
\calC:=\Big\{%
W\in\calH %
\text{ is increasing with}%
\,W\at{\phase}\to\pm1 \text{ as }\phase\to\pm\infty, %
\Big\}
\end{align*}
where the convergence conditions mean that $W$ has asymptotic values
$-1$ and $+1$ in the above sense.
\bigpar
A particular problem in our subsequent analysis is to control the
relative shift between a profile function $W$ and the reference
profile $W_\sh$. For this reason we introduce the set of all pinned
profile functions by
\begin{align}
\notag
\calC_0:=\left\{W\in\calC
\text{ with}%
\begin{array}{lll}
{W}|_{\cointerval{0}{+\infty}}
\text{ takes values in }%
\ccinterval{0}{+1}
\\%
{W}|_{\ocinterval{-\infty}{0}}
\text{ takes values in }%
\ccinterval{-1}{0}
\end{array}%
\right\}.%
\end{align}
Obviously, for each $W\in\calC$ there exists a phase shift
$\phase_0$ such that ${W}\at{\cdot-\phase_0}\in\calC_0$, and this
phase shift is unique if $W$ is strictly increasing. Moreover, we
have $W\in\calC_0$ if and only if $\wt{W}\in\wt{\calC}_0$ with
\begin{align}
\notag
\wt{\calC}_0:=\left\{\wt{W}\in\fspaceL^2\text{ such that }
\begin{array}{lll}
\wt{W}|_{\cointerval{0}{+\infty}} \text{ is decreasing with values
in }\ccinterval{0}{+1}
\\%
\wt{W}|_{\ocinterval{-\infty}{0}}\text{ is decreasing with values in
}\ccinterval{-1}{0}\end{array}\right\}.
\end{align}
Notice that the set $\wt{\calC}_0$ is convex and closed
under weak convergence in $\fspaceL^2$.
%
%
%
\subsection{Variational approach}
%
%
%
The starting point for our variational existence result is the
invariance of $\calC$ under the action of $\calT$. In fact, it is
straightforward that the averaging operator $\calA$ respects both the
monotonicity and asymptotic values of $W$. In addition, since
$\Phi^\prime$ is a monotone map on $\ccinterval{-1}{1}$ with
$\Phi^\prime\at{\pm1}=\pm1$ the set $\calC$ is also
invariant under the pointwise application of $\Phi^\prime$. Notice,
however, that $\calC_0$ is \emph{not} invariant under the action of
$\calT$ since for (continuous) $W$ with $W\at0=0$ in general we have
$\calT\ato{W}\at0\neq0$.
\bigpar%
Next we consider the gradient flow of $\calL$ in $\calH$, that is
\begin{align}
\label{Gradientflow} %
\tfrac{\dint}{\dint{s}}{W}
=%
-\partial\calL\ato{W}=-W+\calT\ato{W}
\end{align}
with flow time $s$. According to Lemma \ref{DerivativeOfL}, the
initial value problem for this $\calH$-valued ODE is well-defined,
and we readily verify that
\begin{align}
\label{Gradientflow.Diss} %
\tfrac{\dint}{\dint{s}}\calL\at{W}
=%
\skp{\partial\calL\at{W}}{\tfrac{\dint}{\dint{s}}{W}}
=%
-\norm{W-\calT\ato{W}}_2^2.
\end{align}
In particular, each stationary point of \eqref{Gradientflow} is a
fixed point of $\calT$.
\bigpar
The key observation is that $\calC$ is an invariant set for
\eqref{Gradientflow}. To see this we introduce the corresponding
Euler scheme with step size $0<\la<1$, that is
\begin{align}
\label{Gradientflow.Scheme} %
W\mapsto{W}-\lambda\partial\calL\ato{W}
=%
\at{1-\lambda}W+\lambda\calT\ato{W}.
\end{align}
This scheme leaves $\calC$ invariant because $W\in\calC$ implies
$\calT\ato{W}\in\calC$, and  since ${\calC}$ is convex and closed in
$\calH$ the claim follows with $\la\to0$. As a consequence we obtain
the following result.
\par
\begin{lemma}
\label{MinimisersAndFronts}
Each local minimiser of $\calL$ in $\calC$ is a solution
to the front equation $W=\calT\ato{W}$.
\end{lemma}
We proceed with two remarks: $(i)$ The front equation can be
considered as the Euler-Lagrange equation for $\calL$ corresponding
to variations $W\rightsquigarrow{W}+\delta\wt{W}$ with arbitrary
$\wt{W}\in\fspaceL^2$. Therefore, the fact that minimisers of
$\calL$ in $\calC$ solve the front equation is not clear a priori
but follows from of the invariance properties of $\calC$ and the
dissipation inequality \eqref{Gradientflow.Diss}. $(ii)$ There is no
analogue to Lemma \ref{MinimisersAndFronts} for local maximisers.
This follows from the next Remark and reflects that $\calC$ is not
invariant under the negative gradient flow of $\calL$.
\begin{remark}
\label{Remark.LocMax}
The shock profile $W_\sh$ is a local maximiser for $\calL$ in
$\calC$ but does not satisfy the front equation.
\end{remark}
\begin{proof}
Let $W\in\calC$ be arbitrary with $W\neq{W}_\sh$. For
$0\leq\eps\leq1$ we define $W_\eps=\at{1-\eps}W_\sh+\eps{W}$ and
note that $W_\eps\in\calC$. Lemma \ref{DerivativeOfL} yields
\begin{align*}
\la:=\frac{\dint}{\dint\eps}\calL\at{W_\eps}|_{\eps=0}
=%
\skp{\partial{\calL}\ato{W_\sh}}{W-W_\sh}=-\int\limits_{\Rset}
\bat{W_\sh-\calT\ato{W_\sh}}\bat{W_\sh-W}\dint\phase.
\end{align*}
On the one hand, we have $\calT\ato{W_\sh}\in\calC$ with
$W_\sh\at\phase<\calT\ato{W_\sh}\at\phase$ for $-1<\phase<0$,
$W_\sh\at\phase>\calT\ato{W_\sh}\at\phase$ for $0<\phase<1$, and
$W_\sh\at\phase=\calT\ato{W_\sh}\at\phase$ for $\abs{\phase}\geq1$.
On the other hand, $W\in\calC$ implies
$W_\sh\at\phase\leq{W}\at\phase$ and
$W_\sh\at\phase\geq{W}\at\phase$ for $\phase<0$ and $\phase>0$,
respectively. From this we conclude $\la\leq0$ and $W\neq{W}_\sh$
gives $\la<0$. Finally, $W_\sh$ does not equal ${\calT}\ato{W_\sh}$
as the latter is continuous.
\end{proof}
Our strategy for proving the existence of fronts is to show that
$\calL$ attains its minimum in $\calC$. To this end we follow the
direct approach, and prove that minimising sequences are precompact.
%
%
\subsection{Bounds for $\calL$ }
%
%
Towards compactness results for minimising sequences for $\calL$ in
${\calC}$, we first prove that corresponding sequences in $\calC_0$
also minimise $\calL$. To this end we start with a technical
result.
\begin{lemma}
\label{Props.AuxLemma} %
Let $U\in\fspaceL^\infty$ be monotone with asymptotic values $u_-$
and $u_+$ for $\phase\to-\infty$ and $\phase\to\infty$,
respectively. Then we have
\begin{align}
\label{Props.AuxLemma.Eqn} %
\int\limits_{\Rset}%
{U\at{\phase+\bar\phase}}-{U\at{\phase}}\dint\phase
=%
\at{u_+-u_-}\bar\phase.
\end{align}
for any phase shift $\bar\phase$.
\end{lemma}
\begin{proof}
For $M>0$ we compute
\begin{align*}
\int\limits_{\Rset}\chi_{\ccinterval{-M}{M}}\at\phase
\Bat{{U\at{\phase+\bar\phase}}-
{U\at{\phase}}}\dint\phase
&=%
\int\limits_{-M+\bar\phase}^{M+\bar\phase}
{U\at{\phase}}\dint\phase
-\int\limits_{-M}^{M}{U\at{\phase}}\dint\phase
=%
\int\limits_{M}^{M+\bar\phase}
{U\at{\phase}}\dint\phase
-\int\limits_{-M}^{-M+\bar\phase}
{U\at{\phase}}\dint\phase
\\&= %
\int\limits_{0}^{\bar\phase}{U\at{M+\phase}}-
{U\at{\phase-M}}\dint\phase.
\end{align*}
Since the integrand on the left hand side has a sign and converges
pointwise, we can pass to the limit $M\to\infty$ by means of the
Monotone Convergence Theorem, and this gives the desired result.
\end{proof}
Notice that the proof of Lemma \ref{Props.AuxLemma} is very close to
that of Lemma \ref{Lemma:JumpCond}, and that the assumptions on $U$
can easily be weakened. For instance, \eqref{Props.AuxLemma.Eqn}
holds also for all functions $U\in\fspace{BV}$ as these can be
written as the difference of two monotone functions with well
defined asymptotic values.
\bigpar
Lemma \ref{Props.AuxLemma} implies that $\calL$ is
invariant under phase shifts, and hence
$\inf\calL|_\calC=\inf\calL|_{\calC_0}$.
\begin{corollary}
\label{Props.InvShifts}%
For all $W\in\calC$ and all phase shifts $\bar\phase$ we have
$\calL\at{W\at{\cdot+\bar\phase}}=\calL\at{W}$.

\end{corollary}
\begin{proof}
We define the monotone functions
\begin{align*}
U_1\at\phase:=%
\inf\limits_{\tilde{\phase}
\geq%
\phase}\tfrac{1}{2}W^2\at{\tilde\phase},\quad
U_2\at\phase:=%
-\inf\limits_{\tilde{\phase}
\geq%
\phase}\Phi\bat{\at{\calA{W}}\at{\tilde\phase}},
\quad{U}_3:=%
\tfrac{1}{2}W^2-U_1,\quad{U}_4:=-\Phi\at{\calA{W}}-U_2,
\end{align*}
and find
\begin{align}
\label{Props.InvShifts.Proof1}%
\calL\at{W\at{\cdot+\bar\phase}}-\calL\at{W}=\sum\limits_{i=1}^4
\int\limits_{\Rset}%
{U_i\at{\phase+\bar\phase}}-{U_i\at{\phase}}\dint\phase
=\bar{\phase}\bat{\Phi\at{-1}-\Phi\at{1}}=0
\end{align}
thanks to Lemma \ref{Props.AuxLemma}.
\end{proof}
Note that the energy condition \cond{E} is essential for Corollary
\ref{Props.InvShifts}. In fact, if it is violated, then \eqref
{Props.InvShifts.Proof1} implies that $\calL$ is unbounded from
below and above.
\bigpar
Next we derive explicit bounds for $\calL\at{W}$ in terms of
$\norm{\wt{W}}_2$ by comparing with the functional
\begin{align}
\label{DefL2}
{\calL_\#}\at{W}:=\int\limits_\Rset
L\pair{W}{{W}}-L\pair{W_\sh}{{W_\sh}}\dint\phase
=\int\limits_{\Rset}g_\Phi\at{W\at\phase}\dint\phase,
\end{align}
where the last identity holds by definition of $g_\Phi$, see
\eqref{Eqn:DefG}, and \cond{E}. In particular, \cond{A} implies
${\calL_\#}\at{W} \geq 0$.
\begin{lemma}
\label{Props.LBounds} %
There exists a constant $c\leq 4$ such that
\begin{align}
\label{Props.LBounds.Eqn1} %
|{\calL_\#}\at{W}-\calL\at{W}|\leq c
\end{align}
for all $W\in\calC$.
\end{lemma}
\begin{proof}
The monotonicity of $W\in\calC$ implies
\begin{align}
\notag
W\at{\phase-\tfrac{1}{2}}
\leq%
{W}\at\phase\leq{W}\at{\phase+\tfrac{1}{2}}
,\quad%
W\at{\phase-\tfrac{1}{2}}
\leq%
\at{\calA{W}}\at\phase\leq{W}\at{\phase+\tfrac{1}{2}}
\end{align}
for all $\phase\in\Rset$, and due to
$\abs{\Phi^{\prime}\at{w}}\leq1$ for all $w\in{\ccinterval{-1}{1}}$
we also have
\begin{align*}
\abs{\Phi\bat{\at{\calA{W}}\at\phase}-\Phi\bat{W\at\phase}}
\leq%
\babs{\at{\calA{W}}\at\phase-W\at\phase}.
\end{align*}
In combination we obtain
\begin{align*}
\abs{\Phi\bat{\at{\calA{W}}\at\phase}-\Phi\bat{W\at\phase}}
\leq{W}\at{\phase+\tfrac{1}{2}}-W\at{\phase-\tfrac{1}{2}}%
\end{align*}
and Lemma \ref{Props.AuxLemma} yields
\begin{align}
\label{Props.LBounds.Est3}%
\int\limits_{\Rset}%
\abs{\Phi\bat{\at{\calA{W}}\at\phase}-
\Phi\bat{W\at\phase}}\dint\phase%
\leq2.
\end{align}
Therefore,
\begin{align*}
\abs{\calL\at{W}-{\calL_\#}\at{W}}
\leq%
\int\limits_{\Rset}%
\abs{\Phi\bat{\at{\calA{W}}\at\phase}-
\Phi\bat{W\at\phase}}\dint\phase
+%
\int\limits_{\Rset}%
\abs{\Phi\bat{\at{\calA{W_\sh}}\at\phase}-\Phi\bat{W_\sh\at\phase}}
\dint\phase\leq{4},
\end{align*}
where we used \eqref{Props.LBounds.Est3} both for $W$ and $W_\sh$.
\end{proof}
The following  lemma is key to our approach and strongly relies on
the signed are condition \cond{A} and the monotonicity of all
profiles $W\in\calC$.
\begin{lemma}
\label{Props.Lemma.Bounds}%
There exist two constants $\ul{c}>0$ and $\ol{c}>0$ such that
\begin{align}
\label{Props.Lemma.Bounds.Eqn1}%
\ul{c}\norm{\wt{W}}^2_2
\leq{\calL_\#}\at{W}\leq\ol{c}\norm{\wt{W}}^2_2
\end{align}
for all $W\in\calC_0$ and $\wt{W}=W_\sh-W$.
\end{lemma}
\begin{proof}
By Assumption \ref{MainAss}, the function $g_\Phi$ is smooth and
positive in $\oointerval{-1}{1}$ with
\begin{align*}
g_\Phi\at{\pm1}=0
,\qquad%
g_\Phi^\prime\at{\pm1}=0
,\qquad%
g_\Phi^{\prime\prime}\at{\pm1}=1-\Phi^{\prime\prime}\at{\pm1},
\end{align*}
and hence we can choose $\ul{c}$ and $\ol{c}$ such that
\begin{align*}
\begin{array}{ccll}
\ul{c}\at{-1-w}^2\leq{g_\Phi\at{w}}\leq\ol{c}\at{-1-w}^2
&\quad&%
\text{for all}&-1\leq{w}\leq0
,\\%
\ul{c}\at{+1-w}^2\leq{g_\Phi\at{w}}\leq\ol{c}\at{+1-w}^2
&\quad&%
\text{for all}&0\leq{w}\leq+1.%
\end{array}
\end{align*}
This implies
$\ul{c}\nat{\wt{W}\at\phase}^2\leq{g_\Phi\at{W\at\phase}}\leq$
$\ol{c}\nat{\wt{W}\at\phase}^2$ for all $\phase\in\Rset$, and
\eqref{Props.Lemma.Bounds.Eqn1} follows thanks to \eqref{DefL2}.
\end{proof}
Notice that the particular values of $\ul{c}$ and $\ol{c}$ depend on
the choice of the pinning point (zero in $\calC_0$), and that there is \emph{no} similar
result for $W\in\calC$. In fact, for any $W\in\calC$ we have
\begin{math}
\calL_\#\at{W\at{\cdot+\bar\phase}}-
\calL_\#\at{W}
=%
\calL\at{W\at{\cdot+\bar\phase}}-\calL\at{W}=0
\end{math}
but $\norm{\wt{W}\at{\cdot+\bar\phase}}_2-\norm{\wt{W}}_2\to\infty$
as $\bar\phase\to\infty$.
\par
Combining  Lemma \ref{Props.LBounds} with Lemma
\ref{Props.Lemma.Bounds} we finally find the desired bounds for
$\calL$.
\begin{corollary}
\label{Props.Corr.Bounds}%
The sub-level sets of $\calL$ in $\calC_0$ are bounded in the
following sense. For each $c_1$ there exists $c_2>0$ such that
$\calL\at{W}\leq{c_1}$ implies $\norm{\wt{W}}_2\leq{c_2}$.
\end{corollary}
%
%
%
%
\subsection{Existence of minimisers}
%
Now we can complete the existence proof for monotone supersonic
fronts. To this end we study weakly convergent minimising sequences
for $\calL$, and exploit the following observations: The limit
$\bar{\gamma}$ of the kinetic energies of the minimising sequence
and the kinetic energy $\gamma_\infty$ of the weak limit satisfy
$\triangle\gamma=\bar{\gamma}-\gamma_\infty\geq 0$.  Consequently,
the weak limit is a minimiser of $\calL$ if and only if the limiting
potential energy difference is less than $\triangle\gamma$. To prove
this, we first observe that the potential energy converges strongly
on compact sets so that we only have to control the potential energy
in the tails of the profile. Second, we show that the loss of
potential energy in the tails is strictly less than
$\triangle\gamma$ provided the shock is supersonic.
\begin{theorem}
\label{EM.ExistenceTheo}  %
$\calL$ attains its minimum in $\calC$. More precisely, $\calL$ is
weakly lower semi-continuous on $\wt{\calC}_0$, and each minimising
sequence has a strongly convergent subsequence.
\end{theorem}
\begin{proof}
We first prove the minimisation in $\calC$ which
consists of the following steps.
\begin{enumerate}
\item
Let $\at{W_n}_n\subset\calC$ be a minimising sequence, i.e.
$\calL\at{W_n}\to\inf\cal{L}|_\calC$. According to Corollary
\ref{Props.InvShifts} we can suppose $W_n\in\calC_0$, and Corollary
\ref{Props.Corr.Bounds} shows that the sequence
$\nat{\wt{W}_n}_n\subset\wt{\calC}_0$ with $\wt{W}_n=W_\sh-W_n$ is
bounded in $\fspaceL^2$. Hence we can extract a (not relabelled)
subsequence such that
\begin{align*}
\wt{W}_n\to\wt{W}_\infty
\quad \text{weakly in }%
\fspaceL^2,\quad\text{and}
\quad%
\ga_n:=\tfrac{1}{2}\norm{\wt{W}_n}^2_2
\to\bar\ga\geq\ga_\infty
:=%
\tfrac{1}{2}\norm{\wt{W}_\infty}_2
\end{align*}
for some weak limit $\wt{W}_\infty$ corresponding to
${W}_\infty=W_\sh-\wt{W}_\infty$. Moreover, we have
$\wt{W}_\infty\in\wt{\calC}_0$ and ${W}_\infty\in{\calC_0}$ as
$\wt{\calC}_0$ is closed under weak convergence. Setting
$U_n:=\wt{W}_n-\wt{W}_\infty={W}_n-{W}_\infty$ we find
\begin{align}
\label{EM.ExistenceTheo.Conv2}%
U_n\to0\quad\text{weakly in }\fspaceL^2,\quad\text{and}\quad
\tfrac{1}{2}\norm{U_n}^2_2\to\bar\ga-\ga_\infty,
\end{align}
which implies
\begin{align}
\label{EM.ExistenceTheo.Conv3}%
\calA{U}_n\to0
\quad\text{and}\quad%
\calA{W}_n\to\calA{W}_\infty
\quad\text{strongly and pointwise in }%
\fspaceL^2\at{\ccinterval{-M}{M}}
\end{align}
for each  $M<\infty$, see Lemma \ref{Props.AvOp}. With these
notations we write
\begin{align*}
\calL\at{W_n}-\calL\at{W_\infty}&=
\int\limits_\Rset\tfrac{1}{2}\Bat{\nat{W_\infty+ U_n}^2-W_\infty^2}
-\Bat{\Phi\nat{\calA{W}_\infty+\calA{U}_n}-
\Phi\at{\calA{W}_\infty}}\dint\phase
\\%
&=I_{1,\,n}+I_{2,\,n}-I_{3,\,n},
\end{align*}
where the $I_{i,\,n}$'s are given by
\begin{align*}
I_{1,\,n}&:=
\tfrac{1}{2}\int\limits_\Rset{U}_n^2\dint\phase,
\\%
I_{2,\,n}&:= \int\limits_\Rset{W_\infty}U_n-
\Phi^\prime\at{\calA{W}_\infty}\calA{U}_n\dint\phase
=\int\limits_\Rset\Bat{{W_\infty}-
\calA\Phi^\prime\at{\calA{W}_\infty}}U_n\dint\phase,
\\%
I_{3,\,n}&:= \int\limits_\Rset
\Phi_{\nat{\calA{W}_\infty}\at\phase}
\bat{\nat{\calA{U}_n}\at\phase}\dint\phase,
\end{align*}
and $\Phi_{\nat{\calA{W}_\infty}\at\phase}\at{\cdot}$ are pointwise
normalised potentials as in Remark \ref{EM.NormPot}.
\item
Towards an estimate for $I_{3,\,n}$, we fix $\eps>0$ sufficiently
small, and $M=M\at{\eps}$ sufficiently large, such that
\begin{align}
\label{EM.ExistenceTheo.Est4}
\sup\limits_{\abs\phase\geq{M}}\abs{W_\sh\at{\phase}-
\at{\calA{W}_\infty}\at{\phase}}\leq\tfrac{1}{2}\eps.
\end{align}
The convergence \eqref{EM.ExistenceTheo.Conv3} provides
$\calA{W}_n\at{\pm{M}}\to\calA{W}_\infty\at{\pm{M}}$,  hence
\begin{align}
\notag
\abs{\at{\calA{W}_\infty-\calA{W}_n}\at{-M}}+
\abs{\at{\calA{W}_\infty-\calA{W}_n}\at{M}}\leq\tfrac{1}{2}\eps
\end{align}
for all sufficiently large $n$, and since $W_n$ is
increasing we conclude that
\begin{align}
\label{EM.ExistenceTheo.Est5} \sup\limits_{\abs\phase\geq{M}}
\abs{W_\sh\at{\phase}-\at{\calA{W}_n}\at{\phase}} \leq\eps.
\end{align}
Combining \eqref{EM.ExistenceTheo.Est4} and
\eqref{EM.ExistenceTheo.Est5} with Remark \eqref{EM.NormPot} we find
\begin{align}
\label{EM.ExistenceTheo.Est1}
0\leq\Phi_{\nat{\calA{W}_\infty}\at\phase}
\bat{\nat{\calA{U}_n}\at\phase}
\leq%
\frac{\zeta_\eps}{2}\nat{\calA{U}_n}\at\phase^2
\end{align}
for $\abs\phase\geq{M}$ and all sufficiently large $n$, where
\begin{align*}
\zeta_\eps:=\sup\Big\{\Phi^{\prime\prime}\at{w}
\;:\;%
w\in\ccinterval{-1}{-1+\eps}\cup\ccinterval{1-\eps}{1}\Big\}.
\end{align*}
With \eqref{EM.ExistenceTheo.Est1} and
$\norm{\calA{U}_n}_2\leq\norm{U_n}_2$ we now estimate
\begin{align*}
0\leq{I}_{3,\,n}&
\leq%
\int\limits_{\abs{\phase}\leq{M}}
\Phi_{\at{\calA{W}_\infty}
\at\phase}\bat{\at{\calA{U}_n}\phase}\dint\phase+
\int\limits_{\abs{\phase}\geq{M}}
\Phi_{\at{\calA{W}_\infty}
\at\phase}\bat{\at{\calA{U}_n}\phase}\dint\phase
\\&\leq%
\tfrac{1}{2}\zeta_2\int\limits_{\abs{\phase}\leq{M}}
\nat{\calA{U}_n}\at\phase^2\dint\phase+
\tfrac{1}{2}\zeta_\eps
\int\limits_{\abs{\phase}\geq{M}}
\nat{\calA{U}_n}\at\phase^2\dint\phase
\\&\leq%
\tfrac{1}{2}{\zeta_2}
\norm{\at{\calA{U_n}}|_{\ccinterval{-M}{M}}}_2^2+
\tfrac{1}{2}\zeta_\eps\norm{U_n}^2_2,
\end{align*}
and exploiting \eqref{EM.ExistenceTheo.Conv3} and
\eqref{EM.ExistenceTheo.Conv2} we obtain
\begin{math}%
\limsup\limits_{n\to\infty}I_{3,\,n}
\leq{\zeta_\eps}\at{\bar\ga-\ga_\infty}
\end{math}. %
Since $\eps$ was chosen arbitrarily, we derive
\begin{align}
\label{EM.ExistenceTheo.Est2}
\limsup\limits_{n\to\infty}I_{3,\,n}&
\leq{\zeta_0}\at{\bar\ga-\ga_\infty},
\end{align}
where
\begin{math}%
\zeta_0=\max\{\Phi^{\prime\prime}\at{-1},\,
\Phi^{\prime\prime}\at{1}\}
\end{math} %
satisfies $0\leq\zeta_0<1$ on account of Assumption \ref{MainAss}.
\item%
Combining \eqref{EM.ExistenceTheo.Est2} with
$I_{1,\,n}\to\bar\gamma-\ga_\infty$ and $I_{2,\,n}\to0$ we find
\begin{align}
\label{EM.ExistenceTheo.IdA}
\lim\limits_{n\to\infty}\calL\at{W_n}
\geq%
\calL\at{W_\infty}+\at{1-\zeta_0}\at{\bar\ga-\ga_\infty}
\geq%
\calL\at{W_\infty},
\end{align}
and since $\at{W_n}_n$ is a minimising sequence, we also have
\begin{align}
\label{EM.ExistenceTheo.IdB}
\calL\at{W_\infty}\geq\lim\limits_{n\to\infty}\calL\at{W_n}.
\end{align}
This shows $\calL\at{W_\infty}
=%
\lim\limits_{n\to\infty}\calL\at{W_n}$, so
${W}_\infty$ is a minimiser of $\calL$ in $\calC$.
\end{enumerate}
Revising the proof so far yields the additional results: $(i)$
$\calL$ is weakly lower semi-continuous on $\wt{\calC}_0$ as
\eqref{EM.ExistenceTheo.IdA} holds even for non-minimising
sequences. $(ii)$ Concerning strongly convergent subsequences,
\eqref{EM.ExistenceTheo.IdA} and \eqref{EM.ExistenceTheo.IdB}
provide $\bar\ga=\ga_\infty$ and hence
$\norm{\wt{W}_n}_2\to\norm{\wt{W}_\infty}_2$, which in turn implies
that the weak convergence $\wt{W}_n-\wt{W}_\infty\to0$ is strong.
\end{proof}
We have now finished the existence proof for action-minimising
fronts with monotone profile. In particular, the second part of
Theorem~\ref{t:main} follows by combining Lemma \ref{Norm.Problem},
Remark \ref{Rem:Regularity}, Lemma \ref{MinimisersAndFronts} and
Theorem \ref{EM.ExistenceTheo}.
\bigpar
We conclude with two remarks about the necessity of conditions
\cond{S} and \cond{A}.
\begin{enumerate}
\item[(i)]
Suppose that $\Phi$ satisfies \cond{E} with strictly concave-convex
$\Phi^\prime$. Then the asymptotic states are subsonic with
$\Phi^{\prime\prime}\at{\pm1}>1$ and $g_\Phi$ is negative in
$\oointerval{-1}{1}$. Using \eqref{DefL2} we show that
\eqref{Props.Lemma.Bounds.Eqn1} holds for negative constants
$\ul{c},\ol{c}<0$, and since \eqref{Props.LBounds.Eqn1} still holds
we conclude that $\calL$ is bounded from above but unbounded from
below. In particular, action minimising fronts cannot exist and each
minimising sequence diverges via $\norm{\wt{W}}_2\to\infty$. On the
other hand, similar to the proof of Theorem \ref{EM.ExistenceTheo}
we can show that $\calL$ attains its maximum in $\calC$. However, we 
cannot conclude that the maximiser solves the front equation,
compare Remark \ref{Remark.LocMax}, and numerical simulations
indicate that the maximiser is not a front.
\item[(ii)]
More generally, suppose \cond{A} fails because there exist $-1<\hat{w}<+1$
with $g_\Phi\at{\hat{w}}<0$, and define the sequence $W_n\in\calC$
with extending plateau at $\hat{w}$ by $W_n\at\phase=\sgn{\phase}$
for $\abs{\phase}\geq{n}$ and $W_n\at\phase=\hat{w}$ for
$\abs{\phase}<{n}$. Then \eqref{DefL2} and
\eqref{Props.LBounds.Eqn1} imply $\calL\at{W_n}\to-\infty$, so
action-minimising fronts cannot exist. Of course, this does not
disprove the existence of local minimisers of $\calL$, but the
numerical simulation in \S\ref{s:num} indicate that these do not
exist.
\end{enumerate}
%
%
%
\section{Qualitative properties and numerical computation of fronts}
\label{s:qualiter}
%
\subsection{Qualitative properties of fronts}\label{s:qual}
%
In this section we discuss some aspects concerning the shape of
front profiles $W$, which hold for any bounded and \emph{monotone}
solution $W=W_\sh-\wt{W}$ to the front equation
$W=\calA\Phi^\prime\at{\calA{W}}$. Recall that such solutions belong
to  $\fspaceC^2$ as noted in Remark~\ref{Rem:Regularity}.
\bigpar
Front profiles are actually strictly monotone for strictly convex
$\Phi$ for the following reason: First we note that there exists at
least one $\phase_0$ such that $W^\prime\at{\phase_0}>0$, and with
$W^\prime\geq0$ this implies $\at{{\calA}W^\prime}\at{\phase}>0$ for
all
$\phase\in\oointerval{\phase_0-\tfrac{1}{2}}{\phase_0+\tfrac{1}{2}}$.
Moreover, the front equation implies
\begin{align*}
{W}^\prime=
\calA\bato{\Phi^{\prime\prime}\at{\calA{W}}{\calA}{W}^\prime},
\end{align*}
and since $\Phi^{\prime\prime}>0$ we infer that
$W^\prime\at{\phase}>0$ for all
$\phase\in\oointerval{\phase_0-1}{\phase_0+1}$. Finally, by
iterating this argument we find $W^{\prime}\at\phase>0$ for all
$\phase\in\Rset$.
\bigpar
We next study the convergence to the asymptotic states.
Heuristically, the convergence to the constant states $w=\pm1$ is
dictated by the linearisation in the asymptotic states, that is
$\wt{W}=\la_\pm\calA^2\wt{W}$, with
$\la_+=\Phi^{\prime\prime}\at{1}$ and
$\la_-=\Phi^{\prime\prime}\at{-1}$ for right and left tails,
respectively. Differentiating twice with respect to $\phase$ we find
the linear advance-delay differential equation
\begin{align*}
\wt{W}'' = \la_\pm\Delta\wt{W},
\end{align*}
with discrete Laplacian
\begin{math}
\nat{\Delta\wt{W}}\at\phase=
\wt{W}\at{\phase+1}+\wt{W}\at{\phase-1}-2\wt{W}{\at\phase}
\end{math}. %
The exponential ansatz $\wt{W}(\vph)=\rme^{-\tau \vph}$
gives the characteristic equation
\begin{align*}
\tau^2 = 2\la_\pm\at{\cosh\at{\tau} -1},
\end{align*}
where decay corresponds to solutions $\tau$ with negative real part
and monotonicity requires $\tau\in\R$. Since supersonic speed
implies $0<\la_\pm<1$ in the normalised potential it is
straightforward to see that the only nonnegative real roots of the
characteristic equation are the double root $\tau=0$ and a simple
root $\tau_\pm>0$.
\par
In conclusion, for monotone $W$ the function $\wt{W}$ is expected to
decay exponentially with rate $\tau_+$ and $\tau_-$ as
$\phase\to+\infty$ and $\phase\to-\infty$, respectively. We prove a
weaker variant of this heuristic result for $\phase\to\infty$, and
mention that an analogous result characterizes the decay for
$\phase\to-\infty$.
\begin{lemma}
\label{DecLemma:ExpTail} %
Let $W$ be a monotone solution to
\eqref{Norm.Problem.Eqn}. Then, for all $\ul{\tau}$,
$\ol{\tau}$ with $\ul\tau<\tau_+<\ol\tau$ there exist positive
constants $\ul{c}$ and $\ol{c}$ such that
\begin{align}
\label{DecLemma:ExpTail.Formula}
\ul{c}\exp\at{-\ol\tau\phase}
\leq%
{\wt{W}\at{\phase}}\leq\ol{c}\exp\at{-\ul\tau\phase}
\end{align}
for all $\phase\geq0$.
\end{lemma}
The recently established existence of centre-stable manifolds for
nonlinear advance-delay differential equations, \cite{Geo08} Theorem
5.1, shows that nonlinear decay is precisely described by the linear
decay except for the centre directions. Since $\tau_+$ is an
algebraically simple eigenvalue there is no secular growth and so
Lemma \ref{DecLemma:ExpTail} in fact implies that
\eqref{DecLemma:ExpTail.Formula} holds for
$\ol{\tau}=\ul{\tau}=\tau_+$ and more precisely that there are
constants $a\neq 0$,  $\nu > 0$ such that
\begin{align*}
\wt{W}(\phase) = a\exp(-\tau_+\phase) + \calO(\exp(-(\tau_++
\nu)\phase).
\end{align*}
Our point, however, is to obtain the decay rate using a direct
approach with a rather simple proof. Not surprisingly, the same
rates of decay describe the tails of small amplitude supersonic
solitons in \cite{FP99} Indeed, Lemma \ref{DecLemma:ExpTail}
requires monotonicity only in the tails and therefore also applies
to all supersonic solitons with monotone tails, as these solve
\eqref{Norm.Problem.Eqn} after a suitable rescaling, compare
\cite{Her09}.
\begin{proof}
We first show the existence of both exponentially decaying lower and
upper bounds for $\wt{W}$ and prove afterwards that $\ul\tau$ and
$\ol\tau$ can be chosen arbitrary close to $\tau_+$.
\begin{enumerate}
\item
By Taylor-expansions of $\Phi$ around $w=1$ and due to
$\nat{\calA\wt{W}}\at\phase\to0$ as $\phase\to\infty$ we find for
each $0<\delta<1$ some $\phase_\delta$ such that
\begin{align}
\label{expTail.Eqn1}
\lambda_+\at{1-\delta}\nat{\calA^2\wt{W}}\at{\phase}
\leq%
\wt{W}\at\phase\leq\lambda_+\at{1+\delta}
\nat{\calA^2\wt{W}}\at{\phase}
\quad\text{for all}\quad\phase\geq\phase_\delta.
\end{align}
Moreover, since $\wt{W}$ is monotonically decreasing for
$\phase\geq0$ we have
\begin{math}
\norm{\wt{W}}_2^2%
\geq\int_{0}^{\phase}\wt{W}
\at{\tilde{\phase}}^2\dint{\tilde{\phase}}
\geq\phase\wt{W}\at{\phase}^2,
\end{math}
and hence
\begin{align}
\label{expTail.Eqn2}
0\leq\wt{W}\at\phase
\leq%
\norm{\wt{W}}_2{\phase}^{-1/2}.
\end{align}
\item
An important building block for the upper bound is the implication
\begin{align}
\label{expTail.Eqn3} %
U\at{\phase}\leq{C}{\phase}^{-1/2}\quad\forall\;\phase>0
\qquad\Longrightarrow\qquad \at{\calA^2{U}}\at{\phase+1}
\leq%
{C}{\phase}^{-1/2}\quad\forall\;\phase>0,
\end{align}
which follows from a direct computation of
$\calA^2\ato{\phase^{-1/2}}$. For fixed $\delta$ and
$\phase\geq\phase_\delta$ we combine the previous results as
follows: \eqref{expTail.Eqn2} and \eqref{expTail.Eqn3} imply
\begin{math}
\nat{\calA^2\wt{W}}\at{\phase+1}
\leq%
{\norm{\wt{W}}_2}{\phase}^{-1/2}
\end{math}
and with \eqref{expTail.Eqn1} we find
\begin{align*}
0
\leq%
\wt{W}\at{\phase+1}
\leq%
\la_+\at{1+\delta}{\norm{\wt{W}}_2}\,{\phase}^{-1/2}.
\end{align*}
Using \eqref{expTail.Eqn3} for $U\at\phase=\wt{W}\at{\phase+1}$ we
obtain
\begin{math}
\nat{\calA^2\wt{W}}\at{\phase+2}
\leq%
\la_+\at{1+\delta}{\norm{\wt{W}}_2}{\phase}^{-1/2},
\end{math}
and \eqref{expTail.Eqn1} yields
\begin{align*}
0\leq\wt{W}\at{\phase+2}
\leq%
\la_+^2\at{1+\delta}^2{\norm{\wt{W}}_2}\,{\phase}^{-1/2}.
\end{align*}
Iterating these arguments gives
\begin{align*}
0\leq\wt{W}\at{\phase+n}
\leq%
\la_+^n\at{1+\delta}^{n}{\norm{\wt{W}}_2}\,{\phase}^{-1/2}\quad
\forall\;\phase\ge\phase_\delta\;\forall\;n\in\Nset,
\end{align*}
and since $\delta$ was arbitrary we have established the upper
estimate in \eqref{DecLemma:ExpTail.Formula} for some $\ul\tau$ with
$\ul\tau>-\ln\lambda_+>0$ and some sufficiently large $\ol{c}$.
\item
Concerning a lower bound, the monotonicity of $\wt{W}$ for
$\phase>0$ implies
\begin{math}
\nat{\calA\wt{W}}\at{\phase+\tfrac{1}{4}}
\geq%
\tfrac{1}{4}\wt{W}\at\phase,
\end{math} %
and therefore
\begin{math}
\nat{\calA^2\wt{W}}\at{\phase+\tfrac{1}{2}}
\geq%
\tfrac{1}{16}\wt{W}\at\phase.
\end{math}
Combination with \eqref{expTail.Eqn1} gives
\begin{align*}
\wt{W}\at{\phase+\tfrac{1}{2}}
\geq\tfrac{1}{16}\at{\lambda_+\at{1-\delta}}\wt{W}\at\phase
\quad
\forall\;\phase \ge \phase_\delta,
\end{align*}
and by iteration we show the existence of an
exponentially decaying lower bound.
\item
Using the previous step we now assume that
\eqref{DecLemma:ExpTail.Formula}
holds for some constants $\ul{c}$ , $\ol{c}$, and some rates
$0<\ul{\tau}<\tau_+<\ol\tau$, and show that these rates can be
improved. In relation to the above characteristic equation we define
the function
\begin{math}
\varrho\at\tau:=2\tau^{-2}\at{\cosh\at\tau-1}
\end{math}
and observe that
$\calA^2\ato{\rme^{-\tau\phase}}=\rho\at\tau\rme^{-\tau\phase}$.
Hence, for all $\tau>0$ and $\phase_\ast>0$ we have
\begin{align}
\label{expTail.Eqn7}%
U\at{\phase}\leq{C}\exp\at{-\tau\phase}
\quad\forall\;\phase>\phase_\ast\qquad\Longrightarrow\qquad
\at{\calA^2{U}}\at{\phase}\leq{C}\varrho\at{\tau}\exp\at{-\tau\phase}
\quad\forall\;\phase>\phase_\ast+1.
\end{align}
Due to \eqref{expTail.Eqn1} and using an iteration argument similar to
the above one we arrive at
\begin{align}
\label{expTail.Eqn8}%
\wt{W}\at{\phase} \leq{C}
\at{\la_+\at{1+\delta}\varrho\at\tau}^n\exp\at{-\tau\phase}
\qquad\forall\;\phase\geq\phase_\delta+n\;\forall\;n\in\Nset,
\end{align}
and since $\wt{W}$ is decreasing this implies
\begin{align*}
\wt{W}\at{\phase_\delta+\phase}
\leq
\tilde{C}\at{\la_+\at{1+\delta}\varrho\at\tau}^\phase\exp\at{-\tau\phase}
\end{align*}
for all $\delta>0$, $\tau>0$, $\phase\geq0$ and some constant
$\tilde{C}$. Therefore we can improve the assumed decay rate
$\ul\tau$ provided that $\la_+ \varrho\at{\ul\tau}<1$, that means
$\ul{\tau}<\tau_+$. More precisely, choosing $\delta$ sufficiently
small and adapting the constant $\ol{c}$ we derive from
\eqref{expTail.Eqn8} that the upper estimate in
\eqref{DecLemma:ExpTail.Formula} also holds for the new rate
\begin{align*}
\ul{\ul\tau}=\ul\tau-\ln\at{\tfrac{1}{2}\at{\la_+\varrho\at{
\ul\tau}+1}}> \ul\tau.
\end{align*}
Since the iteration of the map $\ul\tau\mapsto\ul{\ul\tau}$ yields a
strictly increasing sequence that converges to $\tau_+$, we conclude
that $\ul\tau$ can be chosen arbitrarily close to $\tau_+$. Finally,
reversing the inequality signs on both sides of the implication
\eqref{expTail.Eqn7} we can show that $\ol\tau$ can also be chosen
arbitrarily close to $\tau_+$.
\end{enumerate}
\end{proof}
%
%
\subsection{Numerical computation of fronts}
\label{s:num}
%
%
In view of Theorem \ref{EM.ExistenceTheo} it seems natural to
approximate fronts by using some discrete counterpart of the
gradient flow of $\calL$, as for instance the explicit Euler scheme
\eqref{Gradientflow.Scheme}.
\par
A corresponding numerical scheme on a finite interval is readily
derived and implemented, but from the rigorous point of view the
account of such a scheme is limited: $(i)$ There is no convergence
proof. $(ii)$ Due to the lack of uniqueness results it is not clear
whether or not our existence result covers all fixed points of the
scheme.
\par
Nevertheless, such schemes work very well numerically and provide
moreover some intuitive understanding for why energy conservation,
supersonic front speed, and area condition are necessary for the
existence of action minimising fronts.
\bigpar
We start with functions $\Phi^\prime$ that have exactly one turning
point. A positive example with admissible potential $\Phi$, initial
profile $W_0=W_\sh$, and $\lambda=1$ is plotted in Figure
\ref{NumFig:Ex1}. After a small number of iterations the profile $W$
becomes stationary and satisfies the front equation up to high
order; the same qualitative behaviour can be observed for other
$0\leq\lambda\leq{1}$ and other initial data $W_0\in\calC$.
\begin{figure}[ht!]%
  \centering{%
  \begin{minipage}[c]{0.275\textwidth}%
  \includegraphics[width=\textwidth, draft=\figdraft]%
  {\figfile{ex1_forces}}%
  \end{minipage}%
  \begin{minipage}[c]{0.675\textwidth}%
  \includegraphics[width=0.5\textwidth, draft=\figdraft]%
  {\figfile{ex1_data}}%
  \includegraphics[width=0.5\textwidth, draft=\figdraft]%
  {\figfile{ex1_error}}%
  \end{minipage}%
  }%
  \caption{%
        Example $1$ with $\Phi$ as in Assumption
        \ref{MainAss}: The profiles converge to a front.
  }%
  \label{NumFig:Ex1}%
  \medskip%
  \centering{%
  \begin{minipage}[c]{0.275\textwidth}%
  \includegraphics[width=\textwidth, draft=\figdraft]%
  {\figfile{ex2_forces}}%
  \end{minipage}%
  \begin{minipage}[c]{0.675\textwidth}%
  \includegraphics[width=0.5\textwidth, draft=\figdraft]%
  {\figfile{ex2_data_1}}%
  \includegraphics[width=0.5\textwidth, draft=\figdraft]%
  {\figfile{ex2_data_2}}%
  \end{minipage}%
  }%
  \caption{%
        Example $2$ with $\Phi\at{1}\neq\Phi\at{-1}$:
        The profiles minimise the action through subsequent
        shifts and converge to a constant.
  }%
  \label{NumFig:Ex2}%
  \medskip%
  \centering{%
  \begin{minipage}[c]{0.275\textwidth}%
  \includegraphics[width=\textwidth, draft=\figdraft]%
  {\figfile{ex3_forces}}%
  \end{minipage}%
  \begin{minipage}[c]{0.675\textwidth}%
  \includegraphics[width=0.5\textwidth, draft=\figdraft]%
  {\figfile{ex3_data_1}}%
  \includegraphics[width=0.5\textwidth, draft=\figdraft]%
  {\figfile{ex3_data_2}}%
  \end{minipage}%
  }%
  \caption{%
        Example $3$ with $\Phi^{\prime\prime}\at{\pm1}>1$:
    The profiles converge to a
        constant via an extending plateau.
  }%
  \label{NumFig:Ex3}%
\end{figure}%
\par
In Figure \ref{NumFig:Ex2} we plot the result for the same scheme
applied to a potential that does not satisfy the energy
conservation.  Here the profiles do not converges to a front
solution, but instead form a travelling wave for the iteration:
After some initial iterations we find a stationary profile that is
shifted in each step. Recall that in this case the action functional
$\calL$ is not invariant under shifts, and thus the profiles
successively decrease their action by converging to the asymptotic
state with higher potential energy. More precisely, assuming
$\calT\at{W}=W\at{\cdot+\bar{\phase}}$ and exploiting
\eqref{Props.InvShifts.Proof1} we find
\begin{math}
0\leq\triangle\calL=\calL\at{W}-\calL\at{\calT\ato{W}}=
\bar{\phase}\at{\Phi\at{1}-\Phi\at{-1}},
\end{math} %
and hence $\sgn\at{\bar\phase}=\sgn\bat{\Phi\at{1}-\Phi\at{1}}<1$.
\par
In order to understand the role of supersonic fronts we next
consider potentials that conserve the energy with concave-convex
$\Phi^\prime$. Such potentials are prototypical for subsonic
conservative shocks and have unbounded $\calL$ as discussed at the
end of \S\ref{s:main}. Figure \ref{NumFig:Ex3} illustrates that the
iteration scheme generates a plateau with increasing length and two
counterpropagating travelling waves that connect the plateau to the
asymptotic states. The height of the plateau is the unique solution
to $\bar{w}=\Phi^\prime\at{\bar{w}}$ with $-1<\bar{w}<1$, and as
above we conclude that the decrease in $\calL$ is given by
\begin{math}
\triangle\calL
=%
\bar{\phase}_1\at{\Phi\at{\bar{w}}-\Phi\at{-1}}+
\bar{\phase}_2\at{\Phi\at{1}-\Phi\at{\bar{w}}},
\end{math} %
where $\bar{\phase}_1>0$ and $\bar{\phase}_2<0$ are the pase shifts
for the travelling waves.
\begin{figure}[ht!]%
  \centering{%
  \begin{minipage}[c]{0.275\textwidth}%
  \includegraphics[width=\textwidth, draft=\figdraft]%
  {\figfile{ex4_forces}}%
  \end{minipage}%
  \begin{minipage}[c]{0.675\textwidth}%
  \includegraphics[width=0.5\textwidth, draft=\figdraft]%
  {\figfile{ex4_data}}%
  \includegraphics[width=0.5\textwidth, draft=\figdraft]%
  {\figfile{ex4_error}}%
  \end{minipage}%
  }%
  \caption{%
        Example $4$ with area condition:
        The profiles converge to a front.
  }%
  \label{NumFig:Ex4}%
  \medskip%
  \centering{%
  \begin{minipage}[c]{0.275\textwidth}%
  \includegraphics[width=\textwidth, draft=\figdraft]%
  {\figfile{ex5_forces}}%
  \end{minipage}%
  \begin{minipage}[c]{0.675\textwidth}%
  \includegraphics[width=0.5\textwidth, draft=\figdraft]%
  {\figfile{ex5_data_1}}%
  \includegraphics[width=0.5\textwidth, draft=\figdraft]%
  {\figfile{ex5_data_2}}%
  \end{minipage}%
  }%
  \caption{%
        Example $5$ without area condition:
        The profiles generate an extending plateau.
  }%
  \label{NumFig:Ex5}%
\end{figure}%
\bigpar
Finally, we study two potentials that both have three turning points
in the distance jump of supersonic conservative shock data. The
numerical simulations in Figure \ref{NumFig:Ex4} and Figure
\ref{NumFig:Ex5} indicate that the signed area condition is truly
necessary for the existence of action minimising fronts. If this
condition fails the profiles create again an extending plateau and
two counter-propagating travelling waves, where the plateau height
$\bar{w}$ satisfies $\Phi^\prime\at{\bar{w}}=\bar{w}$ and
$\Phi^{\prime\prime}\at{w}<1$.
\section*{Acknowlegdements}%
This work has been supported in part by the DFG Priority Program
1095 ``Analysis, Modeling and Simulation of Multiscale Problems''
(M.H., J.R.), and the NDNS+ cluster of the Netherlands Organisation for Scientific Research NWO (J.R.).
We thank the reviewers for their insightful comments and for pointing us to additional references.
\providecommand{\bysame}{\leavevmode\hbox to3em{\hrulefill}\thinspace}
\providecommand{\MR}{\relax\ifhmode\unskip\space\fi MR }
\providecommand{\MRhref}[2]{%
  \href{http://www.ams.org/mathscinet-getitem?mr=#1}{#2}
}
\providecommand{\href}[2]{#2}

\end{document}